\documentclass[12pt, reqno, a4paper]{amsart}
\usepackage{amsmath, amssymb, amsthm, amscd}
\usepackage{extarrows}
\usepackage[T2A, T1]{fontenc}
\usepackage{txfonts}	
\usepackage{eucal}
\usepackage[dvips]{color}
\usepackage{multicol}
\usepackage[all]{xy}		
\usepackage{graphicx}
\usepackage{color}
\usepackage{colordvi}
\usepackage{xspace}
\usepackage{tikz}
\usepackage{enumitem}
\usepackage[colorlinks,final,backref=page,hyperindex]{hyperref}
\allowdisplaybreaks % page break for formulars

\newif\ifflabel\flabelfalse
\ifflabel
\else
	
\fi
\newtheorem{theorem}{Theorem}[section]
\newtheorem{lemma}[theorem]{Lemma}
\newtheorem{corollary}[theorem]{Corollary}
\newtheorem{proposition}[theorem]{Proposition}
\theoremstyle{definition}
\newtheorem{definition}[theorem]{Definition}
\newtheorem{example}[theorem]{Example}
\newtheorem{remark}[theorem]{Remark}

\newcommand{\End}{\mathrm{End}}
\newcommand{\Hom}{\mathrm{Hom}}
\newcommand{\id}{\mathrm{id}}

\newcommand{\frakl}{\mathfrak l}
\newcommand{\frakr}{\mathfrak r}

\newcommand{\ad}{\mathrm{ad}}

\newcommand{\delete}[1]{}

\begin{document}
\title[Extended $\mathcal{O}$-operators, extended perm Yang-Baxter equations and related structures]{Extended $\mathcal{O}$-operators, extended perm Yang-Baxter equations and related structures}

%%%%%%%%%%%%%%%%%%%%%%%%%%%%%%%%%%%%%%%%%%%%%%%%%%%%%%%%%%%%%%%%%%%%%%%%%%%%%%%%
%%%%%%%%%%%%%%%%%%%%%%%%%%%%%%%%%%%%%%%%%%%%%%%%%%%%%%%%%%%%%%%%%%%%%%%%%%%%%%%%
%%%%%%%%%%%%%%%%%%%%%%%%%%%%%%%%%%%%%%%%%%%%%%%%%%%%%%%%%%%%%%%%%%%%%%%%%%%%%%%%
\author[S.~Zhang]{Siqi Zhang}
\address{School of Mathematics, North University of China, Taiyuan 030051, China}
\email{2408014201@st.nuc.edu.cn}

\author[Y.~Lin]{Yuanchang Lin$^\ast$}
\thanks{$^\ast$Corresponding author}
\address{School of Mathematics, North University of China, Taiyuan 030051, China}
\email{linyuanchang@mail.nankai.edu.cn}

%%%%%%%%%%%%%%%%%%%%%%%%%%%%%%%%%%%%%%%%%%%%%%%%%%%%%%%%%%%%%%%%%%%%%%%%%%%%%%%%

\subjclass[2020]{
	16T10, % Bialgebras
	16T25, % Yang-Baxter equations
	16W99, % Associative rings and algebras: None of the above, but in this section
	17B38, % Yang-Baxter equations and Rota-Baxter operators
}


\date{\today}

\begin{abstract}
	In this paper, we introduce the notions of extended $\mathcal{O}$-operators on perm algebras, as generalization of $\mathcal{O}$-operators, and the extended perm Yang-Baxter equation.
	We demonstrate that an $\mathcal{O}$-operator of weight $\lambda$ on a perm algebra gives rise to a post-perm algebra, and an extended $\mathcal{O}$-operator yields a new perm algebra structure.
	We also give an equivalent characterization of extended $\mathcal{O}$-operators in terms of ordinary $\mathcal{O}$-operators via the $\pi_\pm$ decomposition.
	The notion of the generalized perm Yang-Baxter equations is introduced, and their relationship with extended $\mathcal{O}$-operators is established.
	The tensor form of extended $\mathcal{O}$-operators leads to the notion of the extended perm Yang-Baxter equation (extended perm-YBE), which generalizes the perm Yang-Baxter equation.
	We establish that a solution of the extended perm-YBE with $(R, \ad)$-invariant skew-symmetric part is characterized by an extended $\mathcal{O}$-operator.
	Furthermore, relationships among extended $\mathcal{O}$-operators, the perm Yang-Baxter equation and the extended perm-YBE are investigated through the framework of quadratic perm algebras and semi-direct product perm algebras.
\end{abstract}

\maketitle

\tableofcontents

%%%%%%%%%%%%%%%%%%%%%%%%%%%%%%%%%%%%%%%%%%%%%%%%%%%%%%%%%%%%%%%%%%%%%%%%%%%%%%%%
%%%%%%%%%%%%%%%%%%%%%%%%%%%%%%%%%%%%%%%%%%%%%%%%%%%%%%%%%%%%%%%%%%%%%%%%%%%%%%%%
%%%%%%%%%%%%%%%%%%%%%%%%%%%%%%%%%%%%%%%%%%%%%%%%%%%%%%%%%%%%%%%%%%%%%%%%%%%%%%%%
%%%%%%%%%%%%%%%%%%%%%%%%%%%%%%%%%%%%%%%%%%%%%%%%%%%%%%%%%%%%%%%%%%%%%%%%%%%%%%%%
%%%%%%%%%%%%%%%%%%%%%%%%%%%%%%%%%%%%%%%%%%%%%%%%%%%%%%%%%%%%%%%%%%%%%%%%%%%%%%%%
%%%%%%%%%%%%%%%%%%%%%%%%%%%%%%%%%%%%%%%%%%%%%%%%%%%%%%%%%%%%%%%%%%%%%%%%%%%%%%%%
%%%%%%%%%%%%%%%%%%%%%%%%%%%%%%%%%%%%%%%%%%%%%%%%%%%%%%%%%%%%%%%%%%%%%%%%%%%%%%%%
\section{Introduction}\label{sec:intro}
The purpose of this paper is to establish extended structures for perm algebras, namely extended $\mathcal{O}$-operators, the extended perm Yang-Baxter equation, and post-perm algebras, following the frameworks developed for Lie algebras, associative and Novikov algebras~\cite{bai2010nonabelian, bai2012O, yu2026extended}.

A perm algebra is an associative algebra satisfying the left-commutative identity \cite{chapoton2002endofoncteur}, and can be obtained from a commutative associative algebra through an averaging operator~\cite{aguiar2000pre}. 
The operads of perm algebras and pre-Lie algebras are Koszul dual to each other, and there exists a Lie algebra structure on the tensor product of a perm algebra and a pre-Lie algebra \cite{chapoton2001pre, ginzburg1994koszul, loday2012algebraic}. 
Recently, \cite{lin2025infinite} lifted this construction to the context of bialgebras and showed that the tensor product of a perm bialgebra and a quadratic pre-Lie algebra can be endowed with a Lie bialgebra structure. 
The theory of perm bialgebras was initiated in \cite{hou2024extending, lin2025infinite}, composed of a perm algebra and a perm coalgebra satisfying certain compatibility conditions.
It is known that the classical Yang-Baxter equation (CYBE) is the fundamental Yang-Baxter-type equation governing Lie bialgebras~\cite{chari1995guide} and the associative Yang-Baxter equation (AYBE) is the fundamental Yang-Baxter-type equation governing associative bialgebras~\cite{bai2010double}. Furthermore, the notion of the perm Yang-Baxter equation (perm-YBE) in perm algebras was introduced as an analog of the CYBE in Lie algebras and the AYBE in associative algebras \cite{bai2010double, chari1995guide}; its symmetric solutions give rise to perm bialgebras, as shown in \cite{lin2025infinite}.

The notion of $\mathcal{O}$-operators, also known as (relative) Rota-Baxter operators, provides a systematic method for constructing solutions to Yang-Baxter-type equations \cite{bai2007unified, bai2010nonabelian, bai2011generalizations, bai2012O}. 
For Lie algebras, extended $\mathcal{O}$-operators were introduced in \cite{bai2010nonabelian} to study double Lie algebra structures and nonabelian generalized Lax pairs. 
These extended operators were shown to generalize Rota-Baxter operators and ordinary $\mathcal{O}$-operators, and their relationship with extended classical Yang-Baxter equations was also established~\cite{bai2010nonabelian, bai2011generalizations, kupershmidt1999classical, semenov1983classical}. 
Analogous results for associative algebras and Novikov algebras were obtained in~\cite{bai2012O, yu2026extended}. 
These developments motivate our investigation of analogous extended structures in the context of perm algebras.

In this paper, we demonstrate the feasibility of extending the aforementioned extended structures to perm algebras. More precisely, we first introduce the notion of $A$-bimodule perm algebras, which encode the actions of $A$ on $V$ via left and right multiplications, together with the perm algebra structure on $V$ itself. 
From such data, we define $\mathcal{O}$-operators of weight $\lambda$ on perm algebras and prove that such operators give rise to post-perm algebra structures. 
We also provide a necessary and sufficient condition for a perm algebra to admit a post-perm algebra structure whose associated perm algebra is the algebra itself.

We then introduce the notion of extended $\mathcal{O}$-operators on perm algebras as generalizations of ordinary $\mathcal{O}$-operators. 
An extended $\mathcal{O}$-operator consists of a linear map $T: V \rightarrow A$ together with an extension $S: V \to A$ satisfying certain conditions, such that the behavior of $T$ is controlled by $S$ and the mass parameter $\kappa$ and $\gamma$.
When $S$ is a balanced $A$-bimodule homomorphism, we establish the $\pi_{\pm}$ decomposition: we show that $T$ is an extended $\mathcal{O}$-operator of weight $\lambda$ with extension $S$ of mass $(-1, \pm\lambda)$ if and only if $T \pm S$ is an $\mathcal{O}$-operator of weight $1$ on $(A, \cdot)$ associated to a new $A$-bimodule perm algebra obtained by $S$. 
This sheds new light on the theory of $\mathcal{O}$-operators through the study of  extended $\mathcal{O}$-operators.

We also introduce the notion of the generalized perm Yang-Baxter equations (generalized perm-YBE), which arise naturally in the study of extended $\mathcal{O}$-operators and perm bialgebras. 
We establish the relationship between extended $\mathcal{O}$-operators and the generalized perm-YBE in the semi-direct product perm algebras.

We further investigate the tensor form of extended $\mathcal{O}$-operators, which leads naturally to the definition of the extended perm Yang-Baxter equations (extended perm-YBE) of mass $\epsilon$. 
This equation generalizes the perm-YBE by introducing a mass term involving the skew-symmetric part of $r$. 
When the skew-symmetric part of $r$ is $(R, \ad)$-invariant, we establish the correspondence between extended $\mathcal{O}$-operators and solutions of the extended perm-YBE. 
In the special case $\kappa = -1$, this equivalence characterizes solutions of the perm-YBE in terms of $\mathcal{O}$-operators.

Moreover, we study the extended perm-YBE in the setting of quadratic perm algebras and semi-direct product perm algebras. 
In a quadratic perm algebra $(A, \cdot, \mathcal{B})$, the nondegenerate skew-symmetric bilinear form $\mathcal{B}$ identifies $A$ with its dual; we show that an extended $\mathcal{O}$-operator on $A$ associated to the adjoint bimodule corresponds to a solution of the extended perm-YBE via $\mathcal{B}$. 
The semi-direct product construction yields a correspondence between extended $\mathcal{O}$-operators on $A$ and solutions of the extended perm-YBE in the semi-direct product perm algebra $(A \ltimes V^*, \bullet)$, generalizing the known relationship between $\mathcal{O}$-operators and the perm-YBE.

The paper is organized as follows.
Section~\ref{sec:prelim} introduce the notions of $A$-bimodule perm algebras, $\mathcal{O}$-operators of weight $\lambda$ and post-perm algebras, and establish the relationship between post-perm algebras and $\mathcal{O}$-operators.
In Section~\ref{sec:extended-o}, we introduce the notion of extended $\mathcal{O}$-operators on perm algebras, and show that extended $\mathcal{O}$-operators induce new perm algebra structures.
Moreover, we give an equivalent characterization of extended $\mathcal{O}$-operators via the $\pi_\pm$ decomposition.
Section~\ref{sec:generalized} introduce the notion of the generalized perm-YBE and establish the relationship between extended $\mathcal{O}$-operators and the generalized perm-YBE.
In Section~\ref{sec:tensor}, we introduce the notion of the extended perm-YBE and investigate the relationship between the extended perm-YBE and extended $\mathcal{O}$-operators.

Throughout this paper, we work over a base field $\mathbf{k}$ of characteristic $0$, and all vector spaces and algebras are assumed to be finite-dimensional.
We adopt the following conventions and notations.
\begin{enumerate}
	\item
	      Let $(A, \diamond)$ be a vector space equipped with a bilinear operation $\diamond: A \otimes A \rightarrow A$.
	      Let $L_{\diamond}(a)$ and $R_{\diamond}(a)$ denote the left and right multiplication operators, that is
	      \begin{equation*}
		      L_{\diamond}(a) b = R_{\diamond}(b) a = a \diamond b, \;\; \forall a, b \in A.
	      \end{equation*}
	      In particular, $\ad_{\diamond}$ is defined by $\ad_{\diamond} = (L_{\diamond} - R_{\diamond})(a)$.
	      We also simply denote them by $L(a)$, $R(a)$ and $\ad(a)$, respectively, without confusion.

	\item
	      Let $V$ be a vector space.
	      Denote the flip operator by $\tau: V \otimes V \rightarrow V \otimes V$, which is defined by
	      \begin{equation*}
		      \tau(u \otimes v) = v \otimes u, \;\; \forall u, v \in V.
	      \end{equation*}

	\item
	      Let $(A, \diamond)$ be a vector space equipped with a bilinear operation $\diamond: A \otimes A \rightarrow A$.
	      Let $r = \sum_{i} a_i \otimes b_i \in A \otimes A$.
	      Set
	      \begin{align*}
		      r_{12} = \sum_{i} a_i \otimes b_i \otimes 1, \;\;
		      r_{13} = \sum_{i} a_i \otimes 1 \otimes b_i, \;\;
		      r_{23} = \sum_{i} 1 \otimes a_i \otimes b_i, \\
		      r_{21} = \sum_{i} b_i \otimes a_i \otimes 1, \;\;
		      r_{31} = \sum_{i} b_i \otimes 1 \otimes a_i, \;\;
		      r_{32} = \sum_{i} 1 \otimes b_i \otimes a_i,
	      \end{align*}
	      where $1$ is the unit if $(A, \diamond)$ is unital or a symbol playing a similar role as the unit for the non-unital cases.
	      Furthermore, define compound symbols such as $r_{12} \diamond r_{13}$ by
	      \begin{equation*}
		      r_{12} \diamond r_{13} = \sum_{i, j} a_i \diamond a_j \otimes b_i \otimes b_j.
	      \end{equation*}

	\item
	      Denote the standard pairing between the dual space $V^*$ and $V$ by
	      \begin{equation*}
		      \langle \ ,\ \rangle: V^* \times V \rightarrow \mathbb{F}, \;\; \langle f, v \rangle := f(v), \;\; \forall f \in V^*, \; v \in V.
	      \end{equation*}

	\item
	      Let $V, W$ be two vector spaces and $T: V \rightarrow W$ be a linear map.
	      Denote the dual map by $T^*: W^* \rightarrow V^*$, which is defined by
	      \begin{equation*}
		      \langle T^*(\xi^*), v \rangle = \langle \xi^*, T(v) \rangle, \;\; \forall v \in V, \; \xi^* \in W^*.
	      \end{equation*}

	\item
	      Let $A, V$ be vector spaces.
	      For a linear map $\zeta: A \rightarrow \End_{\mathbf{k}}(V)$, define a linear map $\zeta^*: A \rightarrow \End(V^*)$ by $\zeta^*(a) = (\zeta(a))^*$, or more explicitly,
	      \begin{equation*}
		      \langle \zeta^*(a)v^*, u\rangle = \langle v^*, \zeta(a) u\rangle, \;\; \forall a \in A, \; u \in V, \; v^* \in V^*.
	      \end{equation*}
\end{enumerate}

%%%%%%%%%%%%%%%%%%%%%%%%%%%%%%%%%%%%%%%%%%%%%%%%%%%%%%%%%%%%%%%%%%%%%%%%%%%%%%%%
%%%%%%%%%%%%%%%%%%%%%%%%%%%%%%%%%%%%%%%%%%%%%%%%%%%%%%%%%%%%%%%%%%%%%%%%%%%%%%%%
%%%%%%%%%%%%%%%%%%%%%%%%%%%%%%%%%%%%%%%%%%%%%%%%%%%%%%%%%%%%%%%%%%%%%%%%%%%%%%%%
\section{\texorpdfstring{$A$}{A}-bimodule perm algebras, post-perm algebras and extended \texorpdfstring{$\mathcal{O}$}{O}-operators}\label{sec:prelim}
In this section, we introduce the notions of $A$-bimodule perm algebras, $\mathcal{O}$-operators of weight $\lambda$ and post-perm algebras, which provide the necessary foundations for the subsequent sections.
Furthermore, we present the relationship between post-perm algebras and $\mathcal{O}$-operators.

%%%%%%%%%%%%%%%%%%%%%%%%%%%%%%%%%%%%%%%%%%%%%%%%%%%%%%%%%%%%%%%%%%%%%%%%%%%%%%%%
\subsection{\texorpdfstring{$A$}{A}-bimodule perm algebras}

\begin{definition}
	A \textbf{perm algebra} $(A, \cdot)$ is a vector space $A$ with a multiplication $\cdot: A \otimes A \to A$ such that
	\begin{equation*}
		a \cdot (b \cdot c)= (a \cdot b) \cdot c = (b \cdot a) \cdot c, \;\; \forall a, b, c \in A. 
	\end{equation*}
	A perm algebra $(A, \cdot)$ is called \textbf{trivial} if $a \cdot b = 0$ for all $a, b \in A$.
\end{definition}

\begin{definition}\label{def:module}
	A \textbf{bimodule} of a perm algebra $(A, \cdot)$ is a triple $(V, \frakl, \frakr)$, where $V$ is a vector space and $\mathfrak{l}, \mathfrak{r}: A \rightarrow \End_{\mathbf{k}}(V)$ are two linear maps satisfying the following equalities for all $a, b \in A$
	\begin{equation}
		\frakl(a \cdot b) = \frakl(a)\frakl(b) = \frakl(b)\frakl(a), \;\;
		\frakr(a \cdot b) = \frakr(b)\frakr(a) = \frakr(b)\frakl(a) = \frakl(a)\frakr(b). \label{eq:repp}
	\end{equation}
	An {\bf $A$-bimodule homomorphism} between bimodules $(V_{1}, \frakl_{1}, \frakr_{1})$ and $(V_{2}, \frakl_{2}, \frakr_{2})$ of $(A, \cdot)$ is a linear map $\varphi: V_{1} \rightarrow V_{2}$ such that
	\begin{equation*}
		\varphi(\frakl_{1}(a)v) = \frakl_{2}(a)(\varphi(v)), \;\;
		\varphi(\frakr_{1}(a)v) = \frakr_{2}(a)(\varphi(v)), \;\;
		\forall a \in A, \; v \in V_{1}.
	\end{equation*}
\end{definition}

Note that if $(V, \frakl, \frakr)$ is a bimodule of $(A, \cdot)$, then $(V^*, \frakl^*, \frakl^* - \frakr^*)$ is also a bimodule of $(A, \cdot)$~\cite{lin2025infinite}.

\begin{example}
	Let $(A, \cdot)$ be a perm algebra.
	Then both $(A, L_{\cdot}, R_{\cdot})$ and $(A^*, L_{\cdot}^*, L_{\cdot}^* - R_{\cdot}^*)$ are bimodules of $(A, \cdot)$.
\end{example}

\begin{definition}
	Let $(A, \cdot)$ and $(V, \cdot_V)$ be perm algebras, and $\frakl, \frakr: A \rightarrow \End_{\mathbf{k}}(V)$ be two linear maps.
	The quadruple $(V, \cdot_V, \frakl, \frakr)$ is called a {\bf $(A, \cdot)$-bimodule perm algebra} (and {\bf $A$-bimodule perm algebra} for short) if $(V, \frakl, \frakr)$ is a bimodule of $(A, \cdot)$ and the following equalities hold for all $a \in A$ and $u, v \in V$:
	\begin{align}
		\frakl(a)(u \cdot_V v) = (\frakl(a)u) \cdot_V v = (\frakr(a)u) \cdot_V v = u \cdot_V (\frakl(a)v), \label{eq:pmc} \\
		\frakr(a)(v \cdot_V u) = \frakr(a)(u \cdot_V v) = u \cdot_V (\frakr(a)v). \label{eq:pmd}
	\end{align}
\end{definition}

\begin{proposition}
	Let $(A, \cdot)$ be a perm algebra, $V$ be a vector space with a binary operation $\cdot_V$, and $\frakl, \frakr: A \rightarrow \End_{\mathbf{k}}(V)$ be two linear maps.
	Define a binary operation $\bullet$ on $A \oplus V$ by
	\begin{align*}
		(a + u) \bullet (b + v) & := a \cdot b + \frakl(a) v + \frakr(b) u + u \cdot_V v, \;\; \forall a, b \in A, \; u, v \in V.
	\end{align*}
	Then $(V, \cdot_V, \frakl, \frakr)$ is an $A$-bimodule perm algebra if and only if $(A \oplus V, \bullet)$ is a perm algebra, which is called the \textbf{semi-direct product of $(A, \cdot)$ and $(V, \cdot_V, \frakl, \frakr)$} and denoted by $(A \ltimes_{\frakl, \frakr} V, \bullet)$.
\end{proposition}
\begin{proof}
	It is straightforward.
\end{proof}

\begin{example}
	Let $(A, \cdot)$ be a perm algebra. Then
	\begin{enumerate}
		\item
		      $(A, \cdot, L_{\cdot}, R_{\cdot})$ is an $A$-bimodule perm algebra.

		\item
		      A bimodule of $(A, \cdot)$ is equivalent to an $A$-bimodule perm algebra equipped with the trivial perm algebra structure.
		      In the sequel, we will adopt this perspective for convenience and always regard a bimodule of $(A, \cdot)$ as such an $A$-bimodule perm algebra.
	\end{enumerate}
\end{example}

%%%%%%%%%%%%%%%%%%%%%%%%%%%%%%%%%%%%%%%%%%%%%%%%%%%%%%%%%%%%%%%%%%%%%%%%%%%%%%%%
\subsection{\texorpdfstring{$\mathcal{O}$}{O}-operators and post-perm algebras}
We first introduce the notion of $\mathcal{O}$-operations and Rota-Baxter operators associated to $A$-bimodule perm algebras, motivated by their counterparts in the context of associative and Lie algebras \cite{bai2010double, bai2010nonabelian, bai2012O, kupershmidt1999classical}.
\begin{definition}\label{def:oo}
	Let $(A, \cdot)$ be a perm algebra, $(V, \cdot_V, \frakl, \frakr)$ be an $A$-bimodule perm algebra and $\lambda \in \mathbf{k}$.
	A linear map $T: V \rightarrow A$ is called an \textbf{$\mathcal{O}$-operator} of weight $\lambda$ on $(A, \cdot)$ associated to $(V, \cdot_V, \frakl, \frakr)$ if
	\begin{equation*}
		T(u) \cdot T(v) = T\big( \frakl(T(u))v + \frakr(T(v))u + \lambda u \cdot_V v \big), \;\; \forall u, v \in V. 
	\end{equation*}
	An $\mathcal{O}$-operator of weight $\lambda$ on $(A, \cdot)$ associated to $(A, \cdot, L_{\cdot}, R_{\cdot})$ is called a {\bf Rota-Baxter operator} of weight $\lambda$ on $(A, \cdot)$, i.e.,
	\begin{equation*}
		T(a) \cdot T(b) = T\big( T(a) \cdot b + a \cdot T(b) + \lambda a \cdot b \big), \;\; \forall a, b \in A.
	\end{equation*}
\end{definition}

\begin{remark}
	Regarding $(V, \frakl, \frakr)$ as an $A$-bimodule equipped with the trivial perm algebra structure, Definition~\ref{def:oo} reduces to the $\mathcal{O}$-operator introduced in \cite[Definition~3.34]{lin2025infinite}.
\end{remark}

\begin{remark}\label{rmk:ool}
	$T$ is an $\mathcal{O}$-operator of weight $\lambda$ on $(A, \cdot)$ associated to $(V, \cdot_V, \frakl, \frakr)$ if and only if $T$ is an $\mathcal{O}$-operator of weight $1$ on $(A, \cdot)$ associated to $(V, \lambda \cdot_V, \frakl, \frakr)$.
	If in addition $\lambda \neq 0$, this is equivalent to $\frac{T}{\lambda}$ being an $\mathcal{O}$-operator of weight $1$ on $(A, \cdot)$ associated to $(V, \cdot_V, \frakl, \frakr)$.
\end{remark}

Next, we introduce the notion of post-perm algebras, which are the perm analog of post-Lie algebras and dendriform trialgebras~\cite{loday2004trialgebras, vallette2007homology}.

\begin{definition}
	A \textbf{post-perm algebra} is a tuple $(A, \circ, \succ, \prec)$, where $(A, \circ)$ is a perm algebra, and $\prec, \succ: A \otimes A \rightarrow A$ are binary operations such that the following equalities hold for all $a, b, c \in A$:
	\begin{eqnarray*}
		(a \succ b + a \prec b + a \circ b) \succ c = a \succ (b \succ c) = b \succ (a \succ c), \\
		c \prec (a \succ b + a \prec b + a \circ b) = (c \prec a) \prec b = (a \succ c) \prec b = a \succ (c \prec b), \\
		a \succ (b \circ c) = (a \succ b) \circ c = (b \prec a) \circ c = b \circ (a \succ c), \\
		(c \circ b) \prec a = (b \circ c) \prec a = b \circ (c \prec a).
	\end{eqnarray*}
\end{definition}

\begin{remark}
	A post-perm algebra with trivial $\circ$ is exactly a pre-perm algebra as defined in~\cite{lin2025infinite}.
\end{remark}

\begin{proposition}\label{prop:pp2p}
	Let $(A, \circ, \succ, \prec)$ be a post-perm algebra.
	Define a binary operation $\cdot: A \otimes A \rightarrow A$ on $A$ by
	\begin{equation*}
		a \cdot b := a \succ b + a \prec b + a \circ b, \;\;
		\forall a, b \in A.
	\end{equation*}
	Then $(A, \cdot)$ is a perm algebra, called the {\bf associated perm algebra} of $(A, \circ, \succ, \prec)$.
	Moreover, $(A, \circ, L_\succ, R_\prec)$ is a $(A, \cdot)$-bimodule perm algebra.
\end{proposition}
\begin{proof}
	A direct computation shows that $(A, \cdot)$ is a perm algebra, and that $(A, \circ, L_\succ, R_\prec)$ is a $(A, \cdot)$-bimodule perm algebra. 
\end{proof}

The relationship between $\mathcal{O}$-operators and post-perm algebras is established in the following results.
\begin{theorem}\label{thm:o2pp}
	Let $(A, \cdot)$ be a perm algebra and $(V, \cdot_V, \frakl, \frakr)$ be an $A$-bimodule perm algebra.
	Let $T: V \rightarrow A$ be an $\mathcal{O}$-operator of weight $\lambda$ on $(A, \cdot)$ associated to $(V, \cdot_V, \frakl, \frakr)$.
	Define the following binary operations on $V$:
	\begin{equation}
		u \circ v = \lambda u \cdot_V v, \;
		u \succ v = \frakl(T(u))v, \;
		u \prec v = \frakr(T(v))u, \;
		\forall u, v \in V. \label{eq:oo2pp}
	\end{equation}
	Then $(V, \circ, \succ, \prec)$ is a post-perm algebra.
	Furthermore, $(V, \circ_{T})$ is a perm algebra and $T$ is a homomorphism of perm algebras from $(V, \circ_{T})$ to $(A, \cdot)$, where $\circ_{T}: A \otimes A \rightarrow A$ is defined by
	\begin{equation}
		u \circ_T v  := \frakl(T(u))v + \frakr(T(v))u + \lambda u \cdot_V v, \;\; \forall u, v \in V. \label{eq:soo2p}
	\end{equation}
\end{theorem}
\begin{proof}
	That $(V, \circ, \succ, \prec)$ is a post-perm algebra can be verified by checking each of the defining relations of a post-perm algebra.
	Then, Proposition~\ref{prop:pp2p} shows that $(V, \circ_{T})$ is a perm algebra, as the associated perm algebra of $(V, \circ, \succ, \prec)$.
	It now follows readily that $T$ is a homomorphism of perm algebras from $(V, \circ_{T})$ to $(A, \cdot)$.
\end{proof}

\begin{proposition}\label{prop:ippeq}
	Let $(A, \cdot)$ be a perm algebra.
	Then there is a post-perm algebra structure on $A$, whose associated perm algebra is $(A, \cdot)$, if and only if there is an invertible $\mathcal{O}$-operator of weight $1$ on $(A, \cdot)$ associated to an $A$-bimodule perm algebra.
\end{proposition}
\begin{proof}
	Suppose that $(A, \circ, \succ, \prec)$ is a post-perm algebra, whose associated perm algebra is $(A, \cdot)$.
	Then $\id: A \rightarrow A$ is an invertible $\mathcal{O}$-operator of weight $1$ on $(A, \cdot)$ associated to the $A$-bimodule perm algebra $(A, \circ, L_\succ, R_\prec)$.
	
	Conversely, suppose that $T: V \rightarrow A$ is an invertible $\mathcal{O}$-operator of weight $1$ on $(A, \cdot)$ associated to an $A$-bimodule perm algebra $(V, \cdot_V, \frakl, \frakr)$.
	Consequently, Theorem~\ref{thm:o2pp} yields a post-perm algebra structure $(V, \circ, \succ, \prec)$ via Eq.~\eqref{eq:oo2pp} with $\lambda = 1$.
	Since $T$ is invertible, this structure can be transformed into $(A, \circ_T, \succ_T, \prec_T)$, explicitly given by
	\begin{equation*}
		a \circ_T b = T(T^{-1}(a) \cdot_V T^{-1}(b)), \;\,
		a \succ_T b = T(\frakl(a) T^{-1}(b)), \;\;
		a \prec_T b = T(\frakr(b) T^{-1}(a)), \;\;
		\forall a, b \in A,
	\end{equation*}
	whose associated perm algebra structure is precisely $(A, \cdot)$.
\end{proof}

\begin{corollary}
	Let $(A, \cdot)$ be a perm algebra, and $T: A \rightarrow A$ be a Rota-Baxter operator of weight $\lambda$ on $(A, \cdot)$.
	Then there is a post-perm algebra structure $(A, \circ, \succ, \prec)$ on $A$ given by
	\begin{align*}
		a \circ b = \lambda a \cdot b, \;
		a \succ b = T(a) \cdot b, \;
		a \prec b = a \cdot T(b), \;
		\forall a, b \in A.
	\end{align*}
	If in addition $T$ is invertible, then there is a post-perm algebra structure $(A, \circ_T, \succ_T, \prec_T)$ on $A$, whose associated perm algebra is $(A, \cdot)$, given by
	\begin{equation*}
		a \circ_T b = \lambda T(T^{-1}(a) \cdot T^{-1}(b)), \;\;
		a \succ_T b = T(a \cdot T^{-1}(b)), \;\;
		a \prec_T b = T(T^{-1}(a) \cdot b), \;\;
		\forall a, b \in A.
	\end{equation*}
\end{corollary}
\begin{proof}
	Since $(A, \cdot, L_{\cdot}, R_{\cdot})$ is an $A$-bimodule perm algebra, Theorem~\ref{thm:o2pp} implies that $(A, \circ, \succ, \prec)$ is a post-perm algebra.
	When $T$ is invertible, this structure can be transformed into $(A, \circ_T, \succ_T, \prec_T)$, whose associated perm algebra is exactly $(A, \cdot)$.
\end{proof}

Theorem~\ref{thm:o2pp} demonstrates that an $\mathcal{O}$-operator yields a new perm algebra structure on $V$.
At the end of this subsection, we provide a necessary and sufficient condition for a general linear map $T$ to induce such a new perm algebra structure.
\begin{lemma}\label{lem:npaeq}
	Let $(A, \cdot)$ be a perm algebra, and $(V, \cdot_V, \frakl, \frakr)$ be an $A$-bimodule perm algebra.
	Let $T: V \rightarrow A$ be a linear map and $\lambda \in \mathbf{k}$.
	Define a binary operation $\circ_T: V \otimes V \rightarrow V$ by Eq.~\eqref{eq:soo2p}.
	Then $(V, \circ_{T})$ is a perm algebra if and only if for all $u, v, w \in V$:
	\begin{align}
		\frakl\big( \mathcal{A}_{T}^{\lambda}(u, v) \big)w = \frakl\big( \mathcal{A}_{T}^{\lambda}(v, u) \big) w = \frakr\big( \mathcal{A}_{T}^{\lambda}(v, w) \big) u, \label{eq:soo}
	\end{align}
	where
	\begin{equation*}
		\mathcal{A}_{T}^{\lambda}(u, v)  = T(u) \cdot T(v) - T(\frakl(T(u))v + \frakr(T(v))u + \lambda u \cdot_V v).
	\end{equation*}
\end{lemma}
\begin{proof}
	For all $u, v, w \in V$, it is easy to check that
	\begin{align*}
		 & (u \circ_T v) \circ_T w - u \circ_T (v \circ_T w) = -\frakl\big( \mathcal{A}_{T}^{\lambda}(u, v) \big)w + \frakr\big( \mathcal{A}_{T}^{\lambda}(v, w) \big) u, \\
		 & (u \circ_T v) \circ_T w - (v \circ_T u) \circ_T w = -\frakl\big( \mathcal{A}_{T}^{\lambda}(u, v) \big)w + \frakl\big( \mathcal{A}_{T}^{\lambda}(v, u) \big) w.     
	\end{align*}
	Hence, $(V, \circ_{T})$ is a perm algebra if and only if Eq.~\eqref{eq:soo} holds.
\end{proof}

%%%%%%%%%%%%%%%%%%%%%%%%%%%%%%%%%%%%%%%%%%%%%%%%%%%%%%%%%%%%%%%%%%%%%%%%%%%%%%%%
%%%%%%%%%%%%%%%%%%%%%%%%%%%%%%%%%%%%%%%%%%%%%%%%%%%%%%%%%%%%%%%%%%%%%%%%%%%%%%%%
%%%%%%%%%%%%%%%%%%%%%%%%%%%%%%%%%%%%%%%%%%%%%%%%%%%%%%%%%%%%%%%%%%%%%%%%%%%%%%%%
%%%%%%%%%%%%%%%%%%%%%%%%%%%%%%%%%%%%%%%%%%%%%%%%%%%%%%%%%%%%%%%%%%%%%%%%%%%%%%%%
\section{Extended \texorpdfstring{$\mathcal{O}$}{o}-operators on perm algebras}\label{sec:extended-o}
In this section, we introduce the notion of extended $\mathcal{O}$-operators on perm algebras as a generalization of ordinary $\mathcal{O}$-operators, following the philosophy of~\cite{bai2010nonabelian, bai2012O, yu2026extended}.
We show that extended $\mathcal{O}$-operators induce new perm algebra structures and give an equivalent characterization of extended $\mathcal{O}$-operators via the $\pi_\pm$ decomposition.

%%%%%%%%%%%%%%%%%%%%%%%%%%%%%%%%%%%%%%%%%%%%%%%%%%%%%%%%%%%%%%%%%%%%%%%%%%%%%%%%
\subsection{From extended \texorpdfstring{$\mathcal{O}$}{O}-operators to new perm algebra structures}

\begin{definition}
	Let $(A, \cdot)$ be a perm algebra, and $(V, \frakl, \frakr)$ be a bimodule of $(A, \cdot)$.
	A linear map $S: V \rightarrow A$ is called {\bf balanced} associated to $(V, \frakl, \frakr)$ if
	\begin{equation}
		\frakl(S(u))v = \frakr(S(v))u, \;\;
		\forall u, v \in V. \label{eq:bal}
	\end{equation}
	If in addition, $S: V \rightarrow A$ is an $A$-bimodule homomorphism from $(V, \frakl, \frakr)$ to $(A, L_{\cdot}, R_{\cdot})$, i.e., 
	\begin{equation}
		S(\frakl(a)u) = a \cdot S(u), \;\; S(\frakr(a)u) = S(u) \cdot a, \;\; \forall a \in A, \; u \in V, \label{eq:ainv}
	\end{equation}
	then $S$ is called a {\bf balanced $A$-bimodule homomorphism} from $(V, \frakl, \frakr)$ to $(A, L_{\cdot}, R_{\cdot})$.
\end{definition}

\begin{lemma}\label{lem:baie}
	Let $(A, \cdot)$ be a perm algebra, and $(V, \cdot_V, \frakl, \frakr)$ be an $A$-bimodule perm algebra.
	Suppose that $S: V \rightarrow A$ is balanced associated to $(V, \frakl, \frakr)$.
	Then
	\begin{equation}
		\frakl(S(u \cdot_V v))w = \big(\frakl(S(u))v\big) \cdot_V w, \;\;
		\frakr(S(u \cdot_V v))w = w \cdot_V \big(\frakr(S(v))u\big) , \;\;
		\forall u,v,w \in V. \label{eq:eqvm}
	\end{equation}
\end{lemma}
\begin{proof}
	For all $u, v, w \in V$, we have
	\begin{align*}
		&\frakl(S(u \cdot_V v))w \overset{\eqref{eq:bal}}{=} \frakr(S(w))(u \cdot_V v) \overset{\eqref{eq:pmd}}{=} u \cdot_V (\frakr(S(w))v) \overset{\eqref{eq:bal}}{=} u \cdot_V (\frakl(S(v))w) \\
		&\overset{\eqref{eq:pmc}}{=} (\frakr(S(v))u) \cdot_V w \overset{\eqref{eq:bal}}{=} (\frakl(S(u))v) \cdot_V w, \\
		&\frakr(S(u \cdot_V v))w \overset{\eqref{eq:bal}}{=} \frakl(S(w))(u \cdot_V v) \overset{\eqref{eq:pmc}}{=} (\frakr(S(w))u) \cdot_V v \overset{\eqref{eq:bal}}{=} (\frakl(S(u))w) \cdot_V v \\
		&\overset{\eqref{eq:pmc}}{=} w \cdot_V (\frakl(S(u))v) \overset{\eqref{eq:bal}}{=} w \cdot_V (\frakr(S(v))u).
	\end{align*}
	The proof is complete.
\end{proof}

\begin{definition}
	Let $(A, \cdot)$ be a perm algebra, and $(V, \cdot_V, \frakl, \frakr)$ be an $A$-bimodule perm algebra.
	Let $T, S: V \to A$ be linear maps, and $\kappa, \gamma, \lambda \in \mathbf{k}$.
	Then $T$ is called an {\bf extended $\mathcal{O}$-operator of weight $\lambda$ with extension $S$ of mass $(\kappa, \gamma)$} on $(A, \cdot)$ associated to $(V, \cdot_V, \frakl, \frakr)$ if
	\begin{equation}
		T(u) \cdot T(v) - T\big(\frakl(T(u))v + \frakr(T(v))u + \lambda u \cdot_V v\big)  = \kappa S(u) \cdot S(v) + \gamma S(u \cdot_V v), \;\; \forall u, v \in V. \label{eq:exto}
	\end{equation}
\end{definition}

\begin{remark}
	The parameters $\kappa, \gamma$ and $\lambda$ in the above definition are introduced so that different cases could be treated uniformly when these parameters vary.
\end{remark}

\begin{example}
	Let $(A, \cdot)$ be a perm algebra.
	\begin{enumerate}

		\item
		      $S = \id: A \rightarrow A$ is a balanced $A$-bimodule homomorphism from $(A, L_{\cdot}, R_{\cdot})$ to $(A, L_{\cdot}, R_{\cdot})$.

		\item
		  	  If $S = 0$, then an extended $\mathcal{O}$-operator of weight $\lambda$ with extension $S$ of mass $(\kappa, \gamma)$ on $(A, \cdot)$ associated to $(V, \cdot_V, \frakl, \frakr)$ reduces to an $\mathcal{O}$-operator of weight $\lambda$ in Definition~\ref{def:oo}.
		  	  
		\item 
			  An extended $\mathcal{O}$-operator of weight $\lambda$ with extension $S$ of mass $(\kappa, \gamma)$ on $(A, \cdot)$ associated to $(V, \frakl, \frakr)$ reduces to an extended $\mathcal{O}$-operator of weight $0$ with extension $S$ of mass $(\kappa, 0)$ on $(A, \cdot)$ associated to $(V, \frakl, \frakr)$.

	\end{enumerate}
\end{example}

\begin{proposition}\label{prop:bae2amp}
	Let $(A, \cdot)$ be a perm algebra, and $(V, \cdot_V, \frakl, \frakr)$ be an $A$-bimodule perm algebra.
	Suppose that $S: V \to A$ is a balanced $A$-bimodule homomorphism from $(V, \frakl, \frakr)$ to $(A, L_{\cdot}, R_{\cdot})$.
	Then $(V, \circ_{+}, \frakl, \frakr)$ (resp. $(V, \circ_{-}, \frakl, \frakr)$) is an $A$-bimodule perm algebra, where $\circ_{+}$ (resp. $\circ_{-}$) is defined by
	\begin{align*}
		u \circ_{+} v = \lambda u \cdot_V v - 2 \frakl(S(u))v \quad \text{(resp. $u \circ_{-} v = \lambda u \cdot_V v + 2 \frakl(S(u))v$)}, \;\; \forall u, v \in V.
	\end{align*}
\end{proposition}
\begin{proof}
	For brevity, we prove only the case $(V, \circ_{+}, \frakl, \frakr)$; the argument for $(V, \circ_{-}, \frakl, \frakr)$ proceeds analogously.
	Let $u, v, w \in V$. Then
	\begin{align*}
		& u \circ_{+} (v \circ_{+} w) = \lambda^2 u \cdot_V (v \cdot_V w) - 2\lambda u \cdot_V (\frakl(S(v))w) - 2\lambda \frakl(S(u))(v \cdot_V w)+ 4 \frakl(S(u))\frakl(S(v))w \\
		&\overset{\hphantom{\;\;}\eqref{eq:pmc}\hphantom{\;\;}}{=} \; \lambda^2 u \cdot_V (v \cdot_V w) - 2\lambda (\frakr(S(v))u) \cdot_V w - 2\lambda \frakl(S(u))(v \cdot_V w)+ 4 \frakl(S(u))\frakl(S(v))w \\
		&\overset{\eqref{eq:bal},\eqref{eq:pmc}}{=} \lambda^2 u \cdot_V (v \cdot_V w) - 4 \lambda \frakl(S(u))(v \cdot_V w)+ 4 \frakl(S(u))\frakl(S(v))w ,\\
		& (u \circ_{+} v) \circ_{+} w = \lambda^2 (u \cdot_V v) \cdot_V w - 2\lambda (\frakl(S(u))v) \cdot_V w - 2\lambda \frakl(S(u \cdot_V v))w + 4 \frakl(S(\frakl(S(u))v))w \\
		& \overset{\eqref{eq:eqvm}, \eqref{eq:pmc}}{=} \lambda^2 (u \cdot_V v) \cdot_V w - 4 \lambda \frakl(S(u)) (v \cdot_V w)  + \frakl(S(\frakl(S(u))v))w \\
		&\overset{\hphantom{\;\;}\eqref{eq:bal}\hphantom{\;\;}}{=} \; \lambda^2 (u \cdot_V v) \cdot_V w - 4 \lambda \frakl(S(u)) (v \cdot_V w) + \frakr(S(w))\frakl(S(u))v \\
		&\overset{\hphantom{\;\;}\eqref{eq:repp}\hphantom{\;\;}}{=} \; \lambda^2 (u \cdot_V v) \cdot_V w - 4 \lambda \frakl(S(u)) (v \cdot_V w) + \frakl(S(u))\frakr(S(w))v \\ 
		&\overset{\hphantom{\;\;}\eqref{eq:bal}\hphantom{\;\;}}{=} \; \lambda^2 (u \cdot_V v) \cdot_V w - 4 \lambda \frakl(S(u)) (v \cdot_V w) + \frakl(S(u))\frakl(S(v))w, \\
		& (v \circ_{+} u) \circ_{+} w = \lambda^2 (v \cdot_V u) \cdot_V w - 2\lambda (\frakl(S(v))u) \cdot_V w - 2\lambda \frakl(S(v \cdot_V u))w + 4 \frakl(S(\frakl(S(v))u))w \\
		& \overset{\eqref{eq:eqvm}, \eqref{eq:bal}}{=} \lambda^2 (v \cdot_V u) \cdot_V w - 4 \lambda (\frakl(S(v))u) \cdot_V w + 4 \frakr(S(w))\frakl(S(v))u \\
		& \overset{\eqref{eq:repp},\eqref{eq:pmc}}{=}  \lambda^2 (v \cdot_V u) \cdot_V w - 4 \lambda (\frakr(S(v))u) \cdot_V w + 4 \frakl(S(v))\frakr(S(w))u \\
		& \overset{\hphantom{\;\;}\eqref{eq:bal}\hphantom{\;\;}}{=} \; \lambda^2 (v \cdot_V u) \cdot_V w - 4 \lambda (\frakl(S(u))v) \cdot_V w + 4 \frakl(S(v))\frakl(S(u))w \\
		& \overset{\eqref{eq:repp},\eqref{eq:pmc}}{=}  \lambda^2 (v \cdot_V u) \cdot_V w - 4 \lambda \frakl(S(u)) (v \cdot_V w ) + 4 \frakl(S(u))\frakl(S(v))w .
	\end{align*}
	Therefore,
	\begin{equation*}
		u \circ_{+} (v \circ_{+} w) = (u \circ_{+} v) \circ_{+} w = (v \circ_{+} u) \circ_{+} w.
	\end{equation*}
	That is, $(V, \circ_{+})$ is a perm algebra.
	One can now readily verify that $(V, \circ_{+}, \frakl, \frakr)$ is an $A$-bimodule perm algebra.
\end{proof}

\begin{theorem}\label{thm:eo2ps}
	Let $(A, \cdot)$ be a perm algebra, $(V, \cdot_V, \frakl, \frakr)$ be an $A$-bimodule perm algebra, and $T, S: V \rightarrow A$ be two linear maps.
	If $S$ is balanced associated to $(V, \frakl, \frakr)$, and $T$ is an extended $\mathcal{O}$-operator of weight $\lambda$ with extension $S$ of mass $(\kappa, \gamma)$ on $(A, \cdot)$ associated to $(V, \cdot_V, \frakl, \frakr)$,
	then $(V, \circ_{T})$ is a perm algebra, where $\circ_T$ is defined by Eq.~\eqref{eq:soo2p}.
\end{theorem}
\begin{proof}
	By Lemma~\ref{lem:npaeq}, it suffices to show that Eq.~\eqref{eq:soo} holds.
	Let $u, v, w \in V$.
	Since $T$ is an extended $\mathcal{O}$-operator of weight $\lambda$ with extension $S$ of mass $(\kappa, \gamma)$ on $(A, \cdot)$ associated $(V, \cdot_V, \frakl, \frakr)$, we have
	\begin{equation*}
		\mathcal{A}_{T}^{\lambda}(u, v) = \kappa S(u) \cdot S(v) + \gamma S(u \cdot_V v).
	\end{equation*}
	Therefore, we have
	\begin{align*}
		&\frakl\big( \mathcal{A}_{T}^{\lambda}(u, v) \big)w - \frakl\big( \mathcal{A}_{T}^{\lambda}(v, u) \big) w \\
		&\overset{\hphantom{\eqref{eq:repp}}}{=}\kappa\frakl(S(u) \cdot S(v))w + \gamma \frakl(S(u \cdot_V v))w - \kappa\frakl(S(v) \cdot S(u))w - \gamma \frakl(S(v \cdot_V u))w \\
		&\overset{\eqref{eq:repp}}{=} \gamma \frakl(S(u \cdot_V v))w - \gamma \frakl(S(v \cdot_V u))w \overset{\eqref{eq:bal}}{=} \gamma \frakr(S(w)) (u \cdot_V v) - \gamma \frakr(S(w)) (v \cdot_V u) \overset{\eqref{eq:pmd}}{=} 0,
	\end{align*}
	and 
	\begin{align*}
		&\frakl\big( \mathcal{A}_{T}^{\lambda}(u, v) \big)w - \frakr\big( \mathcal{A}_{T}^{\lambda}(v, w) \big) u \\
		&\overset{\;\;\hphantom{\eqref{eq:repp}}\;\;}{=} \; \kappa\frakl(S(u) \cdot S(v))w + \gamma \frakl(S(u \cdot_V v))w - \kappa\frakr(S(v) \cdot S(w))u - \gamma \frakr(S(v \cdot_V w))u \\
		&\overset{\;\;\eqref{eq:repp}\;\;}{=} \; \kappa\frakl(S(u)) \frakl(S(v))w  + \gamma \frakl(S(u \cdot_V v))w  - \kappa\frakr(S(w)) \frakr(S(v))u - \gamma \frakr(S(v \cdot_V w))u \\
		&\overset{\eqref{eq:bal},\eqref{eq:eqvm}}{=} \kappa\frakl(S(u))\frakr(S(w))v + \gamma (\frakl(S(u))v) \cdot_V w - \kappa\frakr(S(w)) \frakl(S(u))v - \gamma \frakl(S(u))(v \cdot_V w) \overset{\eqref{eq:repp},\eqref{eq:pmc}}{=} 0.
	\end{align*}
	That is, Eq.~\eqref{eq:soo} holds.
	The proof is complete.
\end{proof}

\begin{corollary}
	Let $(A, \cdot)$ be a perm algebra, and $T, S: A \rightarrow A$ be two linear maps.
	If $S$ is balanced associated to $(A, L_{\cdot}, R_{\cdot})$, and $T: A \rightarrow A$ is an extended $\mathcal{O}$-operator of weight $\lambda$ with extension $S$ of mass $(\kappa, \gamma)$ associated to $(A, \cdot, L_{\cdot}, R_{\cdot})$,
	then $(A, \circ_T)$ is a perm algebra, where
	$\circ_T: A \otimes A \rightarrow A$ is defined by
	\begin{equation*}
		a \circ_T b  := T(a) \cdot b + a \cdot T(b) + \lambda a \cdot b, \;\; \forall a, b \in A.
	\end{equation*}
\end{corollary}
\begin{proof}
	The conclusion follows from Theorem~\ref{thm:eo2ps}.
\end{proof}

%%%%%%%%%%%%%%%%%%%%%%%%%%%%%%%%%%%%%%%%%%%%%%%%%%%%%%%%%%%%%%%%%%%%%%%%%%%%%%%%
\subsection{\texorpdfstring{$\pi_{\pm}$}{pipm} decomposition}

Let $(A, \cdot)$ be a perm algebra, $(V, \cdot_V, \frakl, \frakr)$ be an $A$-bimodule perm algebra, and $\pi_{\pm}: V \rightarrow A$ be two linear maps.
Set
\begin{equation*}
	T:= \frac{1}{2}(\pi_{+} + \pi_{-}), \;\;
	S := \frac{1}{2}(\pi_{+} - \pi_{-}),
\end{equation*}
called the {\bf symmetrizer} and {\bf antisymmetrizer} of $\pi_{\pm}$ respectively.
Note that $\pi_{\pm}$ can be resolved from $T$ and $S$ by $\pi_{\pm} = T \pm S$.

\begin{proposition}\label{prop:pm2p}
	Let $(A, \cdot)$ be a perm algebra, $(V, \cdot_V, \frakl, \frakr)$ be an $A$-bimodule perm algebra, and $\pi_{\pm}: V \rightarrow A$ be two linear maps.
	Let $T, S$ be the symmetrizer and antisymmetrizer of $\pi_{\pm}$ respectively.
	If $S$ is balanced associated to $(V, \frakl, \frakr)$, then $(V, \circ)$ is a perm algebra, where $\circ: V \otimes V \rightarrow V$ is defined by
	\begin{equation*}
		u \circ v := \frakl(\pi_{+}(u))v + \frakr(\pi_{-}(v))u + \lambda u \cdot_V v, \;\; \forall u, v \in V,
	\end{equation*}
	if and only if Eq.~\eqref{eq:soo} holds.
	If in addition $T$ is an $\mathcal{O}$-operator of weight $\lambda$ on $(A, \cdot)$ associated to $(V, \cdot_V, \frakl, \frakr)$, then $(V, \circ)$ is a perm algebra and $T$ is a homomorphism of perm algebras between $(V, \circ)$ and $(A, \cdot)$.
\end{proposition}
\begin{proof}
	Since $S$ is balanced associated to $(V, \frakl, \frakr)$, we have
	\begin{align*}
		u \circ v  &\overset{\hphantom{\eqref{eq:bal}}}{=} \frakl(\pi_{+}(u))v + \frakr(\pi_{-}(v))u + \lambda u \cdot_V v = \frakl((T+S)(u))v + \frakr((T-S)(v))u + \lambda u \cdot_V v \\
		&\overset{\eqref{eq:bal}}{=} \frakl(T(u))v + \frakr(T(v))u + \lambda u \cdot_V v,
	\end{align*}
	for all $u, v \in V$.
	Therefore, by Lemma~\ref{lem:npaeq}, $(V, \circ)$ is a perm algebra if and only if Eq.~\eqref{eq:soo} holds.
	Furthermore, if in addition $T$ is an $\mathcal{O}$-operator of weight $\lambda$ on $(A, \cdot)$ associated to $(V, \cdot_V, \frakl, \frakr)$, then $\mathcal{A}_{T}^{\lambda}(u, v) = 0$ for all $u, v \in V$.
	Hence, $(V, \circ)$ is a perm algebra.
	Since 
	\begin{equation*}
		T(u \circ v) = T(\frakl(T(u))v + \frakr(T(v))u + \lambda u \cdot_V v) = T(u) \cdot T(v), \;\; \forall u, v \in V,
	\end{equation*}
	$T$ is a homomorphism of perm algebras from $(V, \circ)$ to $(A, \cdot)$.
	The proof is complete.
\end{proof}

The following results provide equivalent characterizations of extended $\mathcal{O}$-operators under certain conditions via the $\pi_{\pm}$ decomposition.

\begin{proposition}\label{prop:iphr}
	Let $(A, \cdot)$ be a perm algebra, $(V, \cdot_V, \frakl, \frakr)$ be an $A$-bimodule perm algebra, and $\pi_{\pm}: V \rightarrow A$ be two linear maps.
	Let $T, S$ be the symmetrizer and antisymmetrizer of $\pi_{\pm}$ respectively.
	Suppose that $S$ is an $A$-bimodule homomorphism from $(V, \frakl, \frakr)$ to $(A, L_{\cdot}, R_{\cdot})$.
	Then $T$ is an extended $\mathcal{O}$-operator of weight $\lambda$ with extension $S$ of mass $(-1, \lambda)$ (resp. $(-1, -\lambda)$) on $(A, \cdot)$ associated $(V, \cdot_V, \frakl, \frakr)$ if and only if
	\begin{equation}
		\pi_{+}(u) \cdot \pi_{+}(v)  = \pi_{+}(u \circ_T v) \;\; \text{(resp. $\pi_{-}(u) \cdot \pi_{-}(v) = \pi_{-}(u \circ_T v)$)}, \;\; \forall u, v \in V, \label{eq:piha}
	\end{equation}
	where $\circ_T$ is defined by Eq.~\eqref{eq:soo2p}.
\end{proposition}
\begin{proof}
	We prove only the mass $(-1, \lambda)$ case, and the mass  $(-1, -\lambda)$ case is analogous.
	\begin{align*}
		 & \pi_{+}(u) \cdot \pi_{+}(v) - \pi_{+}(u \circ_T v) = (T+S)(u) \cdot (T+S)(v) - (T+S)(\frakl(T(u))v + \frakr(T(v))u + \lambda u \cdot_V v) \\
		 & \overset{\hphantom{\eqref{eq:ainv}}}{=} T(u) \cdot T(v) + S(u) \cdot S(v) - T(\frakl(T(u))v + \frakr(T(v))u + \lambda u \cdot_V v)                                             \\
		 & \quad + T(u) \cdot S(v) + S(u) \cdot T(v) - S(\frakl(T(u))v + \frakr(T(v))u + \lambda u \cdot_V v)                                       \\
		 & \overset{\eqref{eq:ainv}}{=} T(u) \cdot T(v) + S(u) \cdot S(v) - T(\frakl(T(u))v + \frakr(T(v))u + \lambda u \cdot_V v)                                              \\
		 & \quad + S(\frakl(T(u))v) +  S(\frakr(T(v))u) - S(\frakl(T(u))v + \frakr(T(v))u + \lambda u \cdot_V v)                                  \\
		 & \overset{\hphantom{\eqref{eq:ainv}}}{=} T(u) \cdot T(v) + S(u) \cdot S(v) - T(\frakl(T(u))v + \frakr(T(v))u + \lambda u \cdot_V v) - \lambda S(u \cdot_V v).
	\end{align*}
	Hence, $T$ is an extended $\mathcal{O}$-operator of weight $\lambda$ with extension $S$ of mass $(-1, \lambda)$ on $(A, \cdot)$ associated to $(V, \cdot_V, \frakl, \frakr)$ if and only if Eq.~\eqref{eq:piha} hold.
\end{proof}

\begin{remark}
	$(V, \circ_{T})$ may not be a perm algebra in Proposition~\ref{prop:iphr}.
	However, by Theorem~\ref{thm:eo2ps}, if $S$ is balanced associated to $(V, \frakl, \frakr)$,
	then $(V, \circ_{T})$ is indeed a perm algebra under the conditions of Proposition~\ref{prop:iphr}.
\end{remark}

\begin{theorem}\label{thm:opm}
	Let $(A, \cdot)$ be a perm algebra, $(V, \cdot_V, \frakl, \frakr)$ be an $A$-bimodule perm algebra, and $\pi_{\pm}: V \rightarrow A$ be two linear maps. 
	Let $T, S$ be the symmetrizer and antisymmetrizer of $\pi_{\pm}$ respectively.
	Suppose that $S$ is a balanced $A$-bimodule homomorphism from $(V, \frakl, \frakr)$ to $(A, L_{\cdot}, R_{\cdot})$.
	Then the following conditions are equivalent.
	\begin{enumerate}
		\item\label{it:oeoheo} 
			$T$ is an extended $\mathcal{O}$-operator of weight $\lambda$ with extension $S$ of mass $(-1, \lambda)$ (resp. $(-1, -\lambda)$) on $(A, \cdot)$ associated to $(V, \cdot_V, \frakl, \frakr)$.
		\item\label{it:oeohh}  
			$(V, \circ_{T})$ is a perm algebra, and $\pi_{+}$ (resp. $\pi_{-}$) is a homomorphism of perm algebras between $(V, \circ_{T})$ and $(A, \cdot)$,
			where $\circ_T$ is defined by Eq.~\eqref{eq:soo2p}.
			
		\item\label{it:oeoho}   
			$\pi_{+}$ (resp. $\pi_{-}$) is an $\mathcal{O}$-operator of weight 1 on $(A, \cdot)$ associated to $(V$, $\circ_{+}$, $\frakl$, $\frakr)$ (resp. $(V$, $\circ_{-}$, $\frakl$, $\frakr)$)
			where $(V$, $\circ_{+}$, $\frakl$, $\frakr)$ (resp. $(V$, $\circ_{-}$, $\frakl$, $\frakr)$) is the $A$-bimodule perm algebra obtained in Proposition~\ref{prop:bae2amp}.
	\end{enumerate}
\end{theorem}
\begin{proof}
	We prove only the $(-1, \lambda)$ case; the $(-1, -\lambda)$ case is analogous.
	
	(\ref{it:oeoheo}) $\Longrightarrow$ (\ref{it:oeohh}).
	By Theorem~\ref{thm:eo2ps}, $(V, \circ_{T})$ is a perm algebra. 
	Then Proposition~\ref{prop:iphr} shows that $\pi_{+}$ is a homomorphism of perm algebras.
	
	(\ref{it:oeohh}) $\Longrightarrow$ (\ref{it:oeoheo}).  
	It follows from Proposition~\ref{prop:iphr}.
	
	(\ref{it:oeohh}) $\Longrightarrow$ (\ref{it:oeoho}).
	For all $u, v \in V$, we have
	\begin{align*}
		&\frakl(\pi_{+}(u))v + \frakr(\pi_{+}(v))u + u \circ_{+} v =\frakl(\pi_{+}(u))v + \frakr(\pi_{+}(v))u + \lambda u \cdot_V v - 2 \frakl(S(u))v \\
		&=\frakl(T(u))v + \frakr(T(v))u + \lambda u \cdot_V v + \frakl(S(u))v + \frakr(S(v))u - 2 \frakl(S(u))v \\
		&=\frakl(T(u))v + \frakr(T(v))u + \lambda u \cdot_V v = u \circ_T v.
	\end{align*}
	Therefore, 
	\begin{equation*}
		\pi_{+}(u) \cdot \pi_{+}(v) - \pi_{+}\big(\frakl(\pi_{+}(u))v + \frakr(\pi_{+}(v))u + u \circ_{+} v \big) = \pi_{+}(u) \cdot \pi_{+}(v) - \pi_{+}(u \circ_{T} v) = 0.
	\end{equation*}
	That is, $\pi_{+}$ is an $\mathcal{O}$-operator of weight 1 on $(A, \cdot)$ associated to $(V, \circ_{+}, \frakl, \frakr)$.
	
	(\ref{it:oeoho}) $\Longrightarrow$ (\ref{it:oeohh}).
	Applying Theorem~\ref{thm:o2pp} to $\pi_{+}$, we obtain that $\pi_{+}$ is a homomorphism of perm algebras from $(V, \circ_{\pi_{+}})$ to $(A, \cdot)$, where $\circ_{\pi_{+}}: V \otimes V \to V$ is defined by
	\begin{equation*}
		u \circ_{\pi_{+}} v  := \frakl(\pi_{+}(u))v + \frakr(\pi_{+}(v))u + u \circ_{+} v.
	\end{equation*}
	On the other hand, following the proof of the implication ``(\ref{it:oeohh}) $\Longrightarrow$ (\ref{it:oeoho})'', we obtain 
	\begin{equation*}
		u \circ_{\pi_{+}} v = u \circ_T v, \;\;
		\forall u, v \in V.
	\end{equation*}
	Hence, $\pi_{+}$ is a homomorphism of perm algebras between $(V, \circ_{T})$ and $(A, \cdot)$.
\end{proof}

\begin{corollary}\label{coro:bah2e}
	Let $(A, \cdot)$ be a perm algebra, $(V, \frakl, \frakr)$ be a bimodule of $(A, \cdot)$, and $S: V \rightarrow A$ be a balanced $A$-bimodule homomorphism from $(V, \frakl, \frakr)$ to $(A, L_{\cdot}, R_{\cdot})$.
	Then $(V, \circ_{+}, \frakl, \frakr)$ (resp. $(V, \circ_{-}, \frakl, \frakr)$) is an $A$-bimodule perm algebra, where $\circ_{+}$ (resp. $\circ_{-}$) is defined by
	\begin{align*}
		u \circ_{+} v = - 2 \frakl(S(u))v \quad \text{(resp. $u \circ_{-} v = 2 \frakl(S(u))v$)}, \;\; \forall u, v \in V.
	\end{align*}
	Furthermore, let $T: V \to A$ be a linear map, then the following conditions are equivalent.
	\begin{enumerate}
		\item 
			$T: V \rightarrow A$ is an extended $\mathcal{O}$-operator of weight $0$ with extension $S$ of mass $(-1, 0)$ on $(A, \cdot)$ associated to $(V, \frakl, \frakr)$.
		
		\item 
			$(V, \circ_{T})$ is a perm algebra, and $T+S$ (resp. $T-S$) is a homomorphism of perm algebras between $(V, \circ_{T})$ and $(A, \cdot)$,
			where $\circ_T$ is defined by 
			\begin{equation*}
				u \circ_T v := \frakl(T(u))v + \frakr(T(v))u, \;\; \forall u, v \in V. 
			\end{equation*}
			
		\item 
			$T+S$ (resp. $T-S$) is an $\mathcal{O}$-operator of weight 1 on $(A, \cdot)$ associated to $(V$, $\circ_{+}$, $\frakl$, $\frakr)$ (resp. $(V$, $\circ_{-}$, $\frakl$, $\frakr)$).
	\end{enumerate}
\end{corollary}
\begin{proof}
	Regarding $(V, \frakl, \frakr)$ as an $A$-bimodule perm algebra with trivial perm algebra structure, the conclusions follow from Proposition~\ref{prop:bae2amp} and Theorem~\ref{thm:opm}.
\end{proof}

\begin{corollary}\label{coro:eoqadj}
	Let $(A, \cdot)$ be a perm algebra and $S: A \rightarrow A$ be a balanced $A$-bimodule homomorphism from $(A, L_{\cdot}, R_{\cdot})$ to $(A, L_{\cdot}, R_{\cdot})$.
	Then $(A, \circ_{+}, L_{\cdot}, R_{\cdot})$ (resp. $(A, \circ_{-}, L_{\cdot}, R_{\cdot})$) is an $A$-bimodule perm algebra, where $\circ_{+}$ (resp. $\circ_{-}$) is defined by
	\begin{equation*}
		a \circ_{+} b = \lambda a \cdot b - 2 S(a) \cdot b \quad \text{(resp. $a \circ_{-} b = \lambda a \cdot b + 2 S(a) \cdot b$)}, \;\; \forall a, b \in A.
	\end{equation*}
	Furthermore, let $T: V \to A$ be a linear map, then the following conditions are equivalent.
	\begin{enumerate}
		\item 
			$T$ is an extended $\mathcal{O}$-operator of weight $\lambda$ with extension $S$ of mass $(-1, \lambda)$ (resp. $(-1, -\lambda)$) on $(A, \cdot)$ associated to $(A, \cdot, L_{\cdot}, R_{\cdot})$.
		
		\item 
			$(A, \circ_{T})$ is a perm algebra, and $T+S$ (resp. $T-S$) is a homomorphism of perm algebras from $(A, \circ_{T})$ to $(A, \cdot)$,
			where $\circ_T$ is defined by
			\begin{equation*}
				a \circ_T b := T(a) \cdot b + a \cdot T(b) + \lambda a \cdot b, \;\; \forall a, b \in A. 
			\end{equation*}
			
		\item 
			$T + S$ (resp. $T - S$) is an $\mathcal{O}$-operator of weight 1 on $(A, \cdot)$ associated to $(A$, $\circ_{+}$, $L_{\cdot}$, $R_{\cdot})$ (resp. $(A$, $\circ_{-}$, $L_{\cdot}$, $R_{\cdot})$).
	\end{enumerate}
\end{corollary}
\begin{proof}
	The conclusions follow from Proposition~\ref{prop:bae2amp} and Theorem~\ref{thm:opm}.
\end{proof}

\begin{corollary}
	Let $(A, \cdot)$ be a perm algebra.
	Then $T: A \rightarrow A$ satisfies the following equality
	\begin{equation}
		T(a) \cdot T(b) - T(T(a) \cdot b + a \cdot T(b) + \lambda a \cdot b) = k a \cdot b, \;\; \forall a, b \in A. \label{eq:bpaa}
	\end{equation}
	where $k \in \mathbf{k}$ if and only if $T + \id$ (resp. $T - \id$) is a Rota-Baxter operator of weight $k - 1$ (resp. $-(k-1)$) on $(A, \cdot)$.
\end{corollary}
\begin{proof}
	By Corollary~\ref{coro:eoqadj}, $T$ is an extended $\mathcal{O}$-operator of weight $\lambda$ with extension $S = \id$ of mass $(-1, k+1)$ on $(A, \cdot)$ associated to $(A, \cdot, L_{\cdot}, R_{\cdot})$ if and only if $T + \id$ (resp. $T - \id$) is an $\mathcal{O}$-operator of weight 1 on $(A, \cdot)$ associated to $(A, \circ_{+}, L_{\cdot}, R_{\cdot})$ (resp. $(A, \circ_{-}, L_{\cdot}, R_{\cdot})$), where
	\begin{equation*}
		a \circ_{+} b = (k - 1) a \cdot b \quad\text{(resp. $a \circ_{+} b = -(k - 1) a \cdot b$)}, \;\;
		\forall a, b \in A.
	\end{equation*}
	By Remark~\ref{rmk:ool}, the latter is equivalent to the condition that $T + \id$ (resp. $T - \id$) is an $\mathcal{O}$-operator of weight $k - 1$ (resp. $-(k-1)$) on $(A, \cdot)$ associated to $(A, \cdot, L_{\cdot}, R_{\cdot})$, which completes the proof.
\end{proof}

\begin{remark}
	If $T$ satisfies Eq.~\eqref{eq:bpaa}, then it naturally yields post-perm algebra structures on $A$.
	In fact, $T \pm \id$ is an $\mathcal{O}$-operator of weight $\pm (k-1)$ on $(A, \cdot)$ associated to $(A, \cdot, L_{\cdot}, R_{\cdot})$, by Theorem~\ref{thm:o2pp}, $(A, \circ, \succ, \prec)$ is a post-perm algebra, where $\circ, \succ, \prec: A \otimes A \to A$ are defined by
	\begin{equation*}
		a \circ b = \pm (k-1) a \cdot b, \;
		a \succ b = (T \pm \id)(a) \cdot b, \;
		a \prec b = a \cdot (T \pm \id)(b), \;
		\forall a, b \in A.
	\end{equation*}
\end{remark}

%%%%%%%%%%%%%%%%%%%%%%%%%%%%%%%%%%%%%%%%%%%%%%%%%%%%%%%%%%%%%%%%%%%%%%%%%%%%%%%%
%%%%%%%%%%%%%%%%%%%%%%%%%%%%%%%%%%%%%%%%%%%%%%%%%%%%%%%%%%%%%%%%%%%%%%%%%%%%%%%%
%%%%%%%%%%%%%%%%%%%%%%%%%%%%%%%%%%%%%%%%%%%%%%%%%%%%%%%%%%%%%%%%%%%%%%%%%%%%%%%%
\section{Extended \texorpdfstring{$\mathcal{O}$}{O}-operators and generalized perm Yang-Baxter equations}\label{sec:generalized}

In this section, we introduce the notion of the generalized perm Yang-Baxter equations, which arise naturally in the study of perm bialgebras and extended $\mathcal{O}$-operators.
We then establish the relationship between extended $\mathcal{O}$-operators and the generalized perm Yang-Baxter equations.

\begin{definition}
	(\cite{lin2025infinite})
	Let $(A, \cdot)$ be a perm algebra and $r \in A \otimes A$.
	Then $r$ is called a solution of the \textbf{perm Yang-Baxter equation (perm-YBE)} in $(A, \cdot)$ if 
	\begin{equation*}
		\mathbf{A}(r) :=  r_{13} \cdot r_{23} - r_{12} \cdot r_{23} + r_{13} \cdot r_{12} - r_{12} \cdot r_{13} = 0.
	\end{equation*}
\end{definition}

\begin{definition}[\cite{lin2026quasi}]
	Let $(A, \cdot)$ be a perm algebra and $r \in A \otimes A$.
	Then $r$ is called \textbf{$(R, \ad)$-invariant} if
	\begin{equation}
		(\id \otimes R_\cdot(a) - \ad_\cdot(a) \otimes \id)(r) = 0, \;\; \forall a \in A. \label{eq:liv}
	\end{equation}
\end{definition}

Symmetric solutions of the perm Yang-Baxter equation were studied in \cite{lin2025infinite}, and solutions with $(R, \ad)$-invariant skew-symmetric part were investigated in \cite{lin2026quasi}; both are special cases of $r - \tau(r)$ is $(R, \ad)$-invariant. 
In the present paper, we retain the $(R, \ad)$-invariant condition and study a generalized version of the perm Yang-Baxter equation.

\begin{definition}
	Let $(A, \cdot)$ be a perm algebra.
	An element $r \in A \otimes A$ is called a solution of the \textbf{generalized perm Yang-Baxter equation (generalized perm-YBE)} in $(A, \cdot)$ if for all $a \in A$,
	\begin{align}
		(\id \otimes \id \otimes R_{\cdot}(a) - (\ad_{\cdot}(a) \otimes \id \otimes \id)(\tau \otimes \id))\mathbf{A}(r) &= 0,  \label{eq:gepbe1} \\
		(\id \otimes \id \otimes R_{\cdot}(a))(\id \otimes \id \otimes \id - \tau \otimes \id )\mathbf{A}(r)  &= 0. \label{eq:gepbe2}
	\end{align}
\end{definition}

\begin{lemma}\label{lem:gped}
	Let $(A, \cdot)$ be a perm algebra, and $r \in A \otimes A$.
	Define a linear map $\Delta: A \rightarrow A \otimes A$ by \begin{equation}
		\Delta(a) := (\id \otimes R_\cdot(a) - (L_\cdot - R_\cdot)(a) \otimes \id)(r), \;\; \forall a \in A. \label{eq:pcbda}
	\end{equation}
	If the skew-symmetric part of $r$ is $(R, \ad)$-invariant, then $(A^*, \Delta^*)$ is a perm algebra if and only if $r$ is a solution of the generalized perm-YBE in $(A, \cdot)$.
\end{lemma}
\begin{proof}
	Putting $r = \sum_{i} a_{i} \otimes b_{i} \in A \otimes A$.
	For all $a, b \in A$, we have
	\begin{align*}
		&(\ad_{\cdot}(a) \otimes \id \otimes \id)(\tau \otimes \id) ( (\tau(r)-r)_{12} \cdot r_{23} + (\tau(r)-r)_{12} \cdot r_{13} - r_{13} \cdot (\tau(r)-r)_{12}) \\
		&= -\sum_{j} (\ad_{\cdot}(a) \otimes \id \otimes \id) (\tau \otimes \id)( (\id \otimes R_{\cdot}(a_j) - \ad_{\cdot}(a_j) \otimes \id)(r - \tau(r)) \otimes b_j) = 0, \\
		&(\id \otimes \ad_{\cdot}(a) \otimes \id)( (r - \tau(r))_{12} \cdot r_{23}) + (\ad_{\cdot}(a) \otimes \id \otimes \id)((r-\tau(r))_{12} \cdot r_{13}) \\
		&=\sum_{j} (\id \otimes R_{\cdot}(a \cdot a_j - a_j \cdot a) - \ad_{\cdot}(a \cdot a_j - a_j \cdot a) \otimes \id)(r - \tau(r)) \otimes b_j = 0.
	\end{align*}
	Then by \cite[Lemma~3.28 (d)]{lin2025infinite}, $(A^*, \Delta^*)$ is a perm algebra if and only if Eqs.~\eqref{eq:gepbe1}-\eqref{eq:gepbe2} hold.
	The proof is complete.
\end{proof}

\begin{remark}
	If $r$ is a solution of the generalized perm-YBE in $(A, \cdot)$, whose skew-symmetric part is $(R, \ad)$-invariant, then by \cite[Lemma~3.28]{lin2025infinite} and Lemma~\ref{lem:gped} together with the equalities
	\begin{align*}
		&(R_\cdot(a) \otimes R_\cdot(b))(r-\tau(r)) = (R_\cdot(a) \otimes \id)(\id \otimes R_\cdot(b) - \ad_\cdot(b) \otimes \id)(r-\tau(r)) = 0, \\
		&( (R_\cdot(b \cdot a)-R_\cdot(a \cdot b)) \otimes \id - \ad_\cdot(b) \otimes \ad_\cdot(a) )(r-\tau(r)) \\
		&= \tau(\id \otimes \ad_\cdot(b))(\ad_\cdot(a) \otimes \id - \id \otimes R_\cdot(a))(r - \tau(r)) = 0, 
	\end{align*}
	it follows that $(A, \cdot, \Delta)$ is a perm bialgebra, where $\Delta:A \to A \otimes A$ is defined by Eq.~\eqref{eq:pcbda}.
\end{remark}

Let $A$ be a vector space.
Any $r \in A \otimes A$ can be identified as maps from $A^*$ to $A$, which we denote by $r_{+}: A^* \rightarrow A$ and $r_{-}: A^* \rightarrow A$, respectively and explicitly,
\begin{equation*}
	\langle r_{+}(x^*), y^*\rangle = \langle x^*, r_{-}(y^*)\rangle = \langle r, x^* \otimes y^*\rangle, \;\; \forall x^*, y^* \in A^*.
\end{equation*}
Clearly, $r_{-} = r_{+}^* = (\tau(r))_{+}$.
Note that $r$ is called the \textbf{2-tensor form} of a linear map $\varphi: A^* \rightarrow A$ if $r_{+} = \varphi$.
Moreover, the multiplication on $A^*$ defined by Eqs.~\eqref{eq:pcbda} (as the dual) is given by
\begin{equation}
	x^* \cdot_r y^* = L_{\cdot}^*(r_{+}(x^*)) y^* + \ad_{\cdot}^*(r_{-}(y^*)) x^*, \;\; \forall x^*,y^* \in A^*. \label{eq:pcbdr}
\end{equation}
Writing $r$ as $r = \Lambda + \Theta$ with $\Lambda \in \mathrm{Sym}^2(A)$ and $\Theta \in \mathrm{Alt}^2(A)$, i.e.,
\begin{equation*}
	\Lambda = \frac{r + \tau(r)}{2}, \;\; \Theta = \frac{r - \tau(r)}{2}
\end{equation*}
Immediately, 
\begin{equation*}
	\Lambda_{+} = \Lambda_{-} = \frac{r_{+} + r_{-}}{2}, \;\; \Theta_{+} = -\Theta_{-} = \frac{r_{+} - r_{-}}{2}.
\end{equation*}
That is, $\Lambda_{+}$ and $\Theta_{+}$ are the symmetrizer and antisymmetrizer of $r_{\pm}$ respectively.
Clearly, $\tau(r) = \Lambda - \Theta$.

\begin{lemma}\label{lem:sinvb}
	Let $(A, \cdot)$ be a perm algebra and $r \in A \otimes A$ be skew-symmetric.
	Then the following conditions are equivalent.
	\begin{enumerate}
		\item
		$r$ is $(R, \ad)$-invariant.
		
		\item
		$r_{+}:  A^* \rightarrow A$ is an $A$-bimodule homomorphism from $(A^*, L_{\cdot}^*, L_{\cdot}^* - R_{\cdot}^*)$ to $(A, L_{\cdot}, R_{\cdot})$.
		
		\item
		$r_{+}: A^* \rightarrow A$ is balanced associated to $(A^*, L_{\cdot}^*, L_{\cdot}^* - R_{\cdot}^*)$.
	\end{enumerate}
\end{lemma}
\begin{proof}
	It follows from \cite[Lemma~2.13]{lin2026quasi}.
\end{proof}

\begin{proposition}\label{prop:eocad2gp}
	Let $(A, \cdot)$ be a perm algebra and $r \in A \otimes A$.
	Define a linear map $\Delta: A \to A \otimes A$ by Eq.~\eqref{eq:pcbda}.
	Writing $r$ as $r = \Lambda + \Theta$ with $\Lambda \in \mathrm{Sym}^2(A)$ and $\Theta \in \mathrm{Alt}^2(A)$.
	Suppose that $\Theta$ is $(R, \ad)$-invariant.
	If $\Lambda_{+}$ is an extended $\mathcal{O}$-operator of weight $0$ with extension $\Theta_{+}$ of mass $(\kappa, 0)$ on $(A, \cdot)$ associated to $(A^*, L_{\cdot}^*, L_{\cdot}^* - R_{\cdot}^*)$, then $(A, \Delta^*)$ is a perm algebra on $A^*$ and $r$ is a solution of the generalized perm-YBE in $(A, \cdot)$.
\end{proposition}
\begin{proof}
	By Lemma~\ref{lem:sinvb}, $\Theta_{+}$ is a balanced $A$-homomorphism from $(A^*, L_{\cdot}^*, L_{\cdot}^* - R_{\cdot}^*)$ to $(A, L_{\cdot}, R_{\cdot})$.
	Hence, for all $x^*, y^* \in A^*$,
	\begin{equation*}
		x^* \cdot_r y^* = L_{\cdot}^*(r_{+}(x^*)) y^* + \ad_{\cdot}^*(r_{-}(y^*)) x^* = L_{\cdot}^*(\Lambda_{+}(x^*)) y^* + \ad_{\cdot}^*(\Lambda_{+}(y^*)) x^*.
	\end{equation*}
	Theorem~\ref{thm:eo2ps} then implies that $(A^*, \Delta^* = \cdot_r)$ is a perm algebra.
	Moreover, by Lemma~\ref{lem:gped}, $r$ is a solution of the generalized perm-YBE in $(A, \cdot)$.
\end{proof}

Let $T: V \rightarrow A$ be a linear map.
Through the identification $\Hom(V, A) \cong A \otimes V^* \subset (A \oplus V^*) \otimes (A \oplus V^*)$, we may regard $T$ as an element $r^{T} \in (A \oplus V^*) \otimes (A \oplus V^*)$ of the tensor product.
Concretely, for a basis $\{e_1, \cdots, e_n\}$ of $V$ with the dual basis $\{e^1, \cdots, e^n\}$, we have $r^T = \sum_{i} T(e_i) \otimes e^i$.

\begin{theorem}\label{thm:gl2gp}
	Let $(A, \cdot)$ be a perm algebra, $(V, \frakl, \frakr)$ be a bimodule of $(A, \cdot)$, and $(\mathfrak{A} := A \ltimes_{\frakl^*, \frakl^* - \frakr^*} V^*, \bullet)$ be the semi-direct product of $(A, \cdot)$ and $(V^*, \frakl^*, \frakl^* - \frakr^*)$.
	Let $T: V \rightarrow A$ be a linear map.
	Then $\hat{r} := r^T + \tau(r^T)$ is a symmetric solution of the generalized perm-YBE in $(A \ltimes_{\frakl^*, \frakl^* - \frakr^*} V^*, \bullet)$ if and only if (for all $a \in A$ and $u, v \in V$)
	\begin{align}
		&\mathcal{A}_{T}(\frakl(a)u, v)  = \mathcal{A}_{T}(\frakr(a)u, v) = \mathcal{A}_{T}(u, \frakl(a)v) = a \cdot \mathcal{A}_{T}(u, v), \label{eq:glgp1} \\
		&\mathcal{A}_{T}(u, v) \cdot a = \mathcal{A}_{T}(v, u) \cdot a = \mathcal{A}_{T}(u, \frakr(a)v), \label{eq:glgp2} \\
		&\frakl(\mathcal{A}_{T}(v, u)) w = \frakl(\mathcal{A}_{T}(u, v)) w = \frakr(\mathcal{A}_{T}(v, w)) u, \label{eq:glgp3} 
	\end{align}
	where $\mathcal{A}_{T}$ is defined by
	\begin{equation*}
		\mathcal{A}_{T}(u, v) := T(u) \cdot T(v) - T(\frakl(T(u))v + \frakr(T(v))u), \;\; \forall u, v \in V.
	\end{equation*}
\end{theorem}
\begin{proof}
	Let $\{e_1, \cdots, e_n\}$ be a basis of $V$ and $\{e^1, \cdots, e^n\}$ be the dual basis.
	Then
	\begin{equation*}
		\hat{r} := r^{T} + \tau(r^T) = \sum_{i} \big( T(e_i) \otimes e^i + e^i \otimes T(e_i) \big).
	\end{equation*}
	Thus,
	\begin{align*}
		\mathbf{A}(\hat{r}) & = \sum_{i, j}\big( (\mathcal{A}_{T}(e_i, e_j) - \mathcal{A}_{T}(e_j, e_i)) \otimes e^i \otimes e^j - e^i \otimes e^j \otimes \mathcal{A}_T(e_i, e_j) + e^i \otimes \mathcal{A}_T(e_i, e_j) \otimes e^j \big).
	\end{align*}
	Therefore, $(\id \otimes \id \otimes R_\bullet(a + \xi^*) - (\ad_\bullet(a + \xi^*) \otimes \id \otimes \id)(\tau \otimes \id))\mathbf{A}(\hat{r}) = 0$ holds for all $a \in A$ and $\xi^* \in V^*$ if and only if
	\begin{align*}
		(\id \otimes \id \otimes R_\bullet(a) - (\ad_\bullet(a) \otimes \id \otimes \id)(\tau \otimes \id))\mathbf{A}(\hat{r}) &= 0, \\
		(\id \otimes \id \otimes R_\bullet(e^k) - (\ad_\bullet(e^k) \otimes \id \otimes \id)(\tau \otimes \id))\mathbf{A}(\hat{r}) &= 0,
	\end{align*}
	hold for all $a \in A$ and $e^k \in V^*$ ($k =1,\cdots n$) if and only if
	\begin{align*}
		\mathcal{A}_{T}(e_i, (\frakl - \frakr)(a)e_j) - \mathcal{A}_{T}((\frakl - \frakr)(a)e_j, e_i) - \mathcal{A}_{T}(\frakr(a)e_i, e_j) + \mathcal{A}_{T}(e_j, \frakr(a)e_i) &= 0, \\
		\mathcal{A}_{T}(e_i, e_j) \cdot a - \mathcal{A}_{T}(e_i, \frakr(a)e_j) &= 0, \\
		\mathcal{A}_{T}(e_i, (\frakl - \frakr)(a)e_j) - a \mathcal{A}_{T}(e_i, e_j) + \mathcal{A}_{T}(e_i, e_j) \cdot a &= 0, \\
		-\frakl(\mathcal{A}_{T}(e_j, e_i))e_k + \frakr(\mathcal{A}_{T}(e_j, e_k))e_i &= 0,
	\end{align*}
	hold for all $a \in A$ and $i, j, k \in \{1, \cdots n\}$ if and only if
	\begin{align}
		\mathcal{A}_{T}(u, (\frakl - \frakr)(a)v) - \mathcal{A}_{T}((\frakl - \frakr)(a)v, u) - \mathcal{A}_{T}(\frakr(a)u, v) + \mathcal{A}_{T}(v, \frakr(a)u) &= 0, \label{eq:pf1}\\
		\mathcal{A}_{T}(u, v) \cdot a - \mathcal{A}_{T}(u, \frakr(a)v) &= 0, \label{eq:pf2}\\
		\mathcal{A}_{T}(u, (\frakl - \frakr)(a)v) - a \mathcal{A}_{T}(u, v) + \mathcal{A}_{T}(u, v) \cdot a &= 0, \label{eq:pf3}\\
		-\frakl(\mathcal{A}_{T}(v, u))w + \frakr(\mathcal{A}_{T}(v, w))u &= 0, \label{eq:pf4}
	\end{align}
	hold for all $a \in A$ and $u, v, w \in V$. 
	Similarly, $(\id \otimes \id \otimes R_\bullet(a+\xi^*))(\id \otimes \id \otimes \id - \tau \otimes \id )\mathbf{A}(\hat{r}) = 0$ holds for all $a \in A$ and $\xi^* \in V^*$ if and only if
	\begin{align}
		\mathcal{A}_T(u, v) \cdot a - \mathcal{A}_T(v, u) \cdot a  & = 0, \label{eq:pf5}\\
		\mathcal{A}_T((\frakl - \frakr)(a)u, v) & = 0, \label{eq:pf6}\\
		\frakl(\mathcal{A}_{T}(u, v))w - \frakl(\mathcal{A}_{T}(v, u))w & = 0,\label{eq:pf7}
	\end{align}
	hold for all $a \in A$ and $u, v, w \in V$.
	It is easy to see that Eqs.~\eqref{eq:pf1}-\eqref{eq:pf7} hold if and only if Eqs.~\eqref{eq:glgp1}-\eqref{eq:glgp3} hold.
	The proof is complete.
\end{proof}

\begin{proposition}\label{prop:ex2gp}
	Let $(A, \cdot)$ be a perm algebra, $(V, \cdot_V, \frakl, \frakr)$ be an $A$-bimodule perm algebra, and $T, S: V \to A$ two linear maps.
	Let $(A \ltimes_{\frakl^*, \frakl^* - \frakr^*} V^*, \bullet)$ be the semi-direct product of $(A, \cdot)$ and $(V^*, \frakl^*, \frakl^* - \frakr^*)$.
	If $S$ is a balanced $A$-bimodule homomorphism from $(V, \frakl, \frakr)$ to $(A, L_{\cdot}, R_{\cdot})$, and $T: V \to A$ is an extended $\mathcal{O}$-operator of weight $\lambda$ with extension $S$ of mass $(\kappa, \gamma)$ associated to $(V, \cdot_V, \frakl, \frakr)$,
	then $r^T + \tau(r^T)$ is a symmetric solution of the generalized perm-YBE in $(A \ltimes_{\frakl^*, \frakl^* - \frakr^*} V^*, \bullet)$ if and only if
	\begin{align}
		&\lambda T(\frakl(a)(u \cdot_V v)) = \lambda a \cdot T(u \cdot_V v), \label{eq:eogp1} \\
		&\lambda T(u \cdot_V v) \cdot a = \lambda T(\frakr(a)(u \cdot_V v)), \label{eq:eogp2 } \\
		&\lambda \frakl(T(u \cdot_V v))w = \lambda \frakl(T(v \cdot_V u))w = \lambda \frakr(T(v \cdot_V w))u, \label{eq:eogp3} 
	\end{align}
	for all $u, v, w \in V$.
\end{proposition}
\begin{proof}
	Since $T: V \to A$ is an extended $\mathcal{O}$-operator of weight $\lambda$ with extension $S$ of mass $(\kappa, \gamma)$ associated to $(V, \cdot_V, \frakl, \frakr)$, we have (for all $u, v \in V$)
	\begin{equation*}
		\mathcal{A}_{T}(u, v) := T(u) \cdot T(v) - T(\frakl(T(u))v + \frakr(T(v))u) = \lambda T(u \cdot_V v) + \kappa S(u) \cdot S(v) + \gamma S(u \cdot_V v).
	\end{equation*}
	Since $S$ is a balanced $A$-bimodule homomorphism from $(V, \frakl, \frakr)$ to $(A, L_{\cdot}, R_{\cdot})$, it follows that Eqs.~\eqref{eq:glgp1}-\eqref{eq:glgp3} and Eqs.~\eqref{eq:eogp1}-\eqref{eq:eogp3} are equivalent.
	The conclusion then follows from Theorem~\ref{thm:gl2gp}.
\end{proof}

\begin{corollary}
	Let $(A, \cdot)$ be a perm algebra.
	\begin{enumerate}
		\item
		      Let $(V, \frakl, \frakr)$ be a bimodule of $(A, \cdot)$ and $(A \ltimes_{\frakl^*, \frakl^* - \frakr^*} V^*, \bullet)$ be the semi-direct product of $(A, \cdot)$ and $(V^*, \frakl^*, \frakl^* - \frakr^*)$.
		      If $S$ is a balanced $A$-bimodule homomorphism from $(V, \frakl, \frakr)$ to $(A, L_{\cdot}, R_{\cdot})$, and $T: V \to A$ is an extended $\mathcal{O}$-operator of weight $0$ with extension $S$ of mass $(\kappa, 0)$ associated to $(V, \frakl, \frakr)$,
		      then $r^T + \tau(r^T)$ is a symmetric solution of the generalized perm-YBE in $(A \ltimes_{\frakl^*, \frakl^* - \frakr^*} V^*, \bullet)$.

		\item
			  Let $(A \ltimes_{L_{\cdot}^*, L_{\cdot}^* - R_{\cdot}^*} A^*, \bullet)$ be the semi-direct product of $(A, \cdot)$ and $(A^*, L_{\cdot}^*, L_{\cdot}^* - R_{\cdot}^*)$.
		      If $T: A \rightarrow A$ is an extended $\mathcal{O}$-operator of weight 0 with extension $\id: A \rightarrow A$ of mass $(\kappa, 0)$ associated to $(A, L_{\cdot}, R_{\cdot})$,
		      then $r^T + \tau(r^T)$ is a symmetric solution of the generalized perm-YBE in $(A \ltimes_{L_{\cdot}^*, L_{\cdot}^* - R_{\cdot}^*} A^*, \bullet)$.
	\end{enumerate}
\end{corollary}
\begin{proof}
	Interpreting a bimodule of $(A, \cdot)$ as an $A$-bimodule perm algebra endowed with the trivial perm algebra structure, the conclusion is a direct consequence of Proposition~\ref{prop:ex2gp}.
\end{proof}

%%%%%%%%%%%%%%%%%%%%%%%%%%%%%%%%%%%%%%%%%%%%%%%%%%%%%%%%%%%%%%%%%%%%%%%%%%%%%%%%
%%%%%%%%%%%%%%%%%%%%%%%%%%%%%%%%%%%%%%%%%%%%%%%%%%%%%%%%%%%%%%%%%%%%%%%%%%%%%%%%
%%%%%%%%%%%%%%%%%%%%%%%%%%%%%%%%%%%%%%%%%%%%%%%%%%%%%%%%%%%%%%%%%%%%%%%%%%%%%%%%
\section{Tensor forms of extended \texorpdfstring{$\mathcal{O}$}{O}-operators and extended perm Yang-Baxter equation}\label{sec:tensor}
In this section, we introduce the notion of the extended perm Yang-Baxter equations, which generalize the notion of the perm Yang-Baxter equations and turn out to be a special case of the generalized perm Yang-Baxter equation under the $(R, \ad)$-invariant condition. 
The relationship between the extended perm Yang-Baxter equations and extended $\mathcal{O}$-operators is investigated. 
In particular, the connection between the perm Yang-Baxter equations and extended $\mathcal{O}$-operators is established.

%%%%%%%%%%%%%%%%%%%%%%%%%%%%%%%%%%%%%%%%%%%%%%%%%%%%%%%%%%%%%%%%%%%%%%%%%%%%%%%%
\subsection{Extended perm Yang-Baxter equation}
\label{ss:extended-perm-YBE}

\begin{definition}\label{def:extended-perm-YBE}
	Let $(A, \cdot)$ be a perm algebra and $\epsilon \in \mathbf{k}$.
	We say that $r\in A \otimes A$ is a solution of the \textbf{extended perm Yang-Baxter equation (extended perm-YBE)} of mass $\epsilon$ in $(A, \cdot)$ if
	\begin{equation}
		\mathbf{A}(r)  = r_{13} \cdot r_{23} - r_{12} \cdot r_{23} + r_{13} \cdot r_{12} - r_{12} \cdot r_{13} = \epsilon ( r_{13} - r_{31}) \cdot ( r_{23} - r_{32} ). \label{eq:extended-perm-YBEa}
	\end{equation}
\end{definition}

\begin{remark}
	When $\epsilon = 0$ or $r$ is symmetric, the extended perm-YBE of mass $\epsilon$ coincides with the perm-YBE~\cite{lin2025infinite}.
\end{remark}

\begin{lemma}\label{lem:lreq}
	Let $(A, \cdot)$ be a perm algebra and $r \in A \otimes A$.
	Writing $r$ as $r = \Lambda + \Theta$ with $\Lambda \in \mathrm{Sym}^2(A)$ and $\Theta \in \mathrm{Alt}^2(A)$.
	Suppose that $\Theta$ is $(R, \ad)$-invariant.
	Then $(A^*, \circ_{+}, L_{\cdot}^*, L_{\cdot}^* - R_{\cdot}^*)$ (resp. $(A^*, \circ_{-}, L_{\cdot}^*, L_{\cdot}^* - R_{\cdot}^*)$) is an $A$-bimodule perm algebra, where $\circ_{+}$ (resp. $\circ_{-}$) is defined by
	\begin{equation}
		x^* \circ_{+} y^*  = - 2 L_\cdot^*(\Theta_{+}(x^*)) y^* \quad\;\;\; \text{(resp. $x^* \circ_{-} y^* = 2 L_\cdot^*(\Theta_{+}(x^*)) y^*$)}. \label{eq:dmd}
	\end{equation}
	Furthermore, the following conditions are equivalent.
	\begin{enumerate}
		\item 
			$\Lambda_{+}$ is an extended $\mathcal{O}$-operator of weight $0$ with extension $\Theta_{+}$ of mass $(-1, 0)$ on $(A, \cdot)$ associated to $(A^*, L_{\cdot}^*, L_{\cdot}^* - R_{\cdot}^*)$.
			
		\item 
			$(A^*, \cdot_r)$ is a perm algebra, and $r_{+}$ (resp. $r_{-}$) is a homomorphism of perm algebras between $(A^*, \cdot_{r})$ and $(A, \cdot)$,
			where $\cdot_{r}$ is defined by Eq.~\eqref{eq:pcbdr}.
			
		\item 
			$r_{+}$ (resp. $r_{-}$) is an $\mathcal{O}$-operator of weight $1$ on $(A, \cdot)$ associated to the $A$-bimodule perm algebra $(A^*, \circ_{+}, L_{\cdot}^*, L_{\cdot}^* - R_{\cdot}^*)$ (resp. $(A^*, \circ_{-}, L_{\cdot}^*, L_{\cdot}^* - R_{\cdot}^*)$).
	\end{enumerate}
\end{lemma}
\begin{proof}
	By Lemma~\ref{lem:sinvb}, $\Theta_{+}: A^* \rightarrow A$ is a balanced $A$-bimodule homomorphism from $(A^*, L_{\cdot}^*, L_{\cdot}^* - R_{\cdot}^*)$ to $(A, L_{\cdot}, R_{\cdot})$.
	Moreover, for all $x^*, y^* \in A^*$, we have
	\begin{equation*}
		L_{\cdot}^*(\Lambda_{+}(x^*))y^* + \ad_{\cdot}^*(\Lambda_{+}(y^*))x^* = L^*(r_{+}(x^*)) y^* + \ad^*(r_{-}(y^*)) x^* = x^* \cdot_{r} y^*.
	\end{equation*}
	The desired conclusion now follows immediately from Corollary~\ref{coro:bah2e}.
\end{proof}

\begin{theorem}\label{thm:extended-perm-YBE-o}
	Let $(A, \cdot)$ be a perm algebra and $r \in A \otimes A$.
	Writing $r$ as $r = \Lambda + \Theta$ with $\Lambda \in \mathrm{Sym}^2(A)$ and $\Theta \in \mathrm{Alt}^2(A)$.
	Suppose that $\Theta$ is $(R, \ad)$-invariant.
	Then $r$ is a solution of the extended perm-YBE of mass $\frac{\kappa+1}{4}$ in $(A, \cdot)$ if and only if $\Lambda_{+}: A^* \rightarrow A$ is an extended $\mathcal{O}$-operator of weight $0$ with extension $\Theta_{+}$ of mass $(\kappa, 0)$ on $(A, \cdot)$ associated to $(A^*, L_{\cdot}^*, L_{\cdot}^* - R_{\cdot}^*)$.
\end{theorem}
\begin{proof}
	For all $x^*, y^*, z^* \in A^*$, note that $\Theta_{+}: A^* \rightarrow A$ is a balance $A$-bimodule homomorphism from $(A^*, L_{\cdot}^*, L_{\cdot}^* - R_{\cdot}^*)$ to $(A, L_{\cdot}, R_{\cdot})$, we have
	\begin{align*}
		 & r_{+}(x^{*}) \cdot r_{+}(y^{*}) - r_{+}(L_{\cdot}^*(r_{+}(x^{*})) y^{*}+\ad_{\cdot}^{*}(r_{-}(y^{*})) x^{*})                                                                        \\
		 & = (\Lambda_{+}+\Theta_{+})(x^{*}) \cdot (\Lambda_{+}+\Theta_{+})(y^{*}) - (\Lambda_{+}+\Theta_{+})(L_{\cdot}^{*}((\Lambda_{+}+\Theta_{+})(x^{*})) y^{*}+\ad_{\cdot}^{*}((\Lambda_{+}-\Theta_{+})(y^{*})) x^{*}) \\
		 & = (\Lambda_{+}+\Theta_{+})(x^{*}) \cdot (\Lambda_{+}+\Theta_{+})(y^{*}) - (\Lambda_{+}+\Theta_{+})(L_{\cdot}^{*}(\Lambda_{+}(x^{*})) y^{*}+\ad_{\cdot}^{*}(\Lambda_{+}(y^{*})) x^{*}) \\
		 & = \Lambda_{+}(x^*) \cdot \Lambda_{+}(y^*) - \Lambda_{+}\big(L_{\cdot}^*(\Lambda_{+}(x^*))y^* + \ad_{\cdot}^*(\Lambda_{+}(y^*))x^*\big) + \Theta_{+}(x^*) \cdot \Theta_{+}(y^*) \\
		 & \quad + \Theta_{+}(x^*) \cdot \Lambda_{+}(y^*) + \Lambda_{+}(x^*) \cdot \Theta_{+}(y^*) - \Theta_{+}(L_{\cdot}^{*}(\Lambda_{+}(x^{*})) y^{*}+\ad_{\cdot}^{*}(\Lambda_{+}(y^{*})) x^{*})  \\
		 & = \Lambda_{+}(x^*) \cdot \Lambda_{+}(y^*) - \Lambda_{+}\big(L_{\cdot}^*(\Lambda_{+}(x^*))y^* + \ad_{\cdot}^*(\Lambda_{+}(y^*))x^*\big) + \Theta_{+}(x^*) \cdot \Theta_{+}(y^*).
	\end{align*}
	Therefore,
	\begin{align*}
		 & \langle \Lambda_{+}(x^*) \cdot \Lambda_{+}(y^*) - \Lambda_{+}\big(L_{\cdot}^*(\Lambda_{+}(x^*))y^* + \ad_{\cdot}^*(\Lambda_{+}(y^*))x^*\big) - \kappa \Theta_{+}(x^*) \cdot \Theta_{+}(y^*), z^* \rangle \\
		 & =\langle r_{+}(x^{*}) \cdot r_{+}(y^{*}) - r_{+}(L_{\cdot}^{*}(r_{+}(x^{*})) y^{*}+\ad_{\cdot}^{*}(r_{-}(y^{*})) x^{*}) - (\kappa+1) \Theta_{+}(x^*) \cdot \Theta_{+}(y^*), z^*\rangle                             \\
		 & =\langle r_{13} \cdot r_{23} - r_{12} \cdot r_{23} + r_{13} \cdot r_{12} - r_{12} \cdot r_{13} - \frac{(\kappa+1)}{4}( r_{13} - r_{31}) \cdot ( r_{23} - r_{32} ), x^* \otimes y^* \otimes z^*\rangle.
	\end{align*}
	Hence, $r$ is a solution of the extended perm-YBE of mass $\frac{\kappa+1}{4}$ in $(A, \cdot)$ if and only if $\Lambda_{+}: A^* \rightarrow A$ is an extended $\mathcal{O}$-operator of weight $0$ with extension $\Theta_{+}$ of mass $(\kappa, 0)$ on $(A, \cdot)$ associated to $(A^*, L_{\cdot}^*, L_{\cdot}^* - R_{\cdot}^*)$.
\end{proof}

\begin{corollary}\label{coro:trr}
	Let $(A, \cdot)$ be a perm algebra and $r \in A \otimes A$ with $(R, \ad)$-invariant skew-symmetric part.
	Then $r$ is a solution of the extended perm-YBE of mass $\epsilon$ in $(A, \cdot)$ if and only if $\tau(r)$ is a solution of the extended perm-YBE of mass $\epsilon$ in $(A, \cdot)$.
\end{corollary}
\begin{proof}
	Write $r$ as $r = \Lambda + \Theta$, where $\Lambda \in \mathrm{Sym}^2(A)$ and $\Theta \in \mathrm{Alt}^2(A)$.
	Since $\Theta$ is $(R, \ad)$-invariant, so is $-\Theta$.
	Note that $\Lambda_{+}: A^* \rightarrow A$ is an extended $\mathcal{O}$-operator of weight $0$ with extension $\Theta_{+}$ of mass $(\kappa, 0)$ on $(A, \cdot)$ associated to $(A^*, L_{\cdot}^*, L_{\cdot}^* - R_{\cdot}^*)$ if and only if $\Lambda_{+}: A^* \rightarrow A$ is an extended $\mathcal{O}$-operator of weight $0$ with extension $-\Theta_{+}$ of mass $(\kappa, 0)$ on $(A, \cdot)$ associated to $(A^*, L_{\cdot}^*, L_{\cdot}^* - R_{\cdot}^*)$.
	Hence, by Theorem~\ref{thm:extended-perm-YBE-o}, this is equivalent to both $r$ and $\tau(r) = \Lambda - \Theta$ being solutions of the extended perm-YBE of mass $\frac{\kappa+1}{4}$ in $(A, \cdot)$, respectively.
	The proof is complete.
\end{proof}

\begin{corollary}
	Let $(A, \cdot)$ be a perm algebra and $r \in A \otimes A$.
	If $r$ is a solution of the extended perm-YBE of mass $\epsilon$ in $(A, \cdot)$ whose skew-symmetric part is $(R, \ad)$-invariant, then $r$ is a solution of the generalized perm-YBE in $(A, \cdot)$.
\end{corollary}
\begin{proof}
	Write $r$ as $r = \Lambda + \Theta$, where $\Lambda \in \mathrm{Sym}^2(A)$ and $\Theta \in \mathrm{Alt}^2(A)$.
	Since $\Theta$ is $(R, \ad)$-invariant, Theorem~\ref{thm:extended-perm-YBE-o} implies that $\Lambda_{+}: A^* \rightarrow A$ is an extended $\mathcal{O}$-operator of weight $0$ with extension $\Theta_{+}$ of mass $(4\epsilon - 1, 0)$ on $(A, \cdot)$ associated to $(A^*, L_\cdot^*, L_\cdot^* - R_\cdot^*)$.
	Proposition~\ref{prop:eocad2gp} then shows that $r$ is a solution of the generalized perm-YBE in $(A, \cdot)$.
\end{proof}

\begin{proposition}\label{prop:peqc}
	Let $(A, \cdot)$ be a perm algebra and $r \in A \otimes A$.
	Writing $r$ as $r = \Lambda + \Theta$ with $\Lambda \in \mathrm{Sym}^2(A)$ and $\Theta \in \mathrm{Alt}^2(A)$.
	Suppose that $\Theta$ is $(R, \ad)$-invariant.
	Then the following conditions are equivalent.
	\begin{enumerate}
		\item\label{it:reqd1}
		      $r$ is a solution of the perm-YBE in $(A, \cdot)$

		\item\label{it:reqd2}
		      $\Lambda_{+}$ is an extended $\mathcal{O}$-operator of weight $0$ with extension $\Theta_{+}$ of mass $(-1, 0)$ associated to $(A^*, L_{\cdot}^*, L_{\cdot}^* - R_{\cdot}^*)$.

		\item\label{it:reqd3}
		      $r_{+}$ (resp. $r_{-}$) is an $\mathcal{O}$-operator of weight $1$ on $(A, \cdot)$ associated to the $A$-bimodule perm algebra $(A^*, \circ_{+}, L_{\cdot}^*, L_{\cdot}^* - R_{\cdot}^*)$ (resp. $(A^*, \circ_{-}, L_{\cdot}^*, L_{\cdot}^* - R_{\cdot}^*)$), where $\circ_{+}$ (resp. $\circ_{-}$) is defined by Eq.~\eqref{eq:dmd}.

		\item\label{it:reqd4}
		      $(A^*, \cdot_r)$ is a perm algebra, where $\cdot_r$ is defined by Eq.~\eqref{eq:pcbdr}, and $r_{+}$ (resp. $r_{-}$) is homomorphism from $(A^*, \cdot_r)$ to $(A, \cdot)$.
	\end{enumerate}
\end{proposition}
\begin{proof}
	(\ref{it:reqd1}) $\Longleftrightarrow$ (\ref{it:reqd2}).
	Note that $r$ is a solution of the perm-YBE in $(A, \cdot)$ if and only if $r$ is a solution of the extended perm-YBE of mass $0$ in $(A, \cdot)$.
	Then by Theorem~\ref{thm:extended-perm-YBE-o}, we show that $r$ is a solution of the perm-YBE in $(A, \cdot)$ if and only if $\Lambda_{+}$ is an extended $\mathcal{O}$-operator of weight $0$ with extension $\Theta_{+}$ of mass $(-1, 0)$ associated to $(A^*, L_{\cdot}^*, L_{\cdot}^* - R_{\cdot}^*)$.

	(\ref{it:reqd2}) $\Longleftrightarrow$ (\ref{it:reqd3}) $\Longleftrightarrow$ (\ref{it:reqd4}).
	It is a consequent conclusion of Lemma~\ref{lem:lreq}.
\end{proof}

\begin{remark}
	The equivalence stated in (\ref{it:reqd1}) and (\ref{it:reqd4}) was shown in \cite[Proposition~ 2.10]{lin2026quasi} previously, we provide an alternative recovery of it based on extended $\mathcal{O}$-operators.
\end{remark}

\begin{remark}
	In particular, if $r$ is symmetric, i.e., $\Theta=0$, then $r$ is a solution of the perm-YBE in $(A, \cdot)$ if and only if $r_{+}$ is an $\mathcal{O}$-operator of weight 0 on $(A, \cdot)$ associated to $(A^*, L_{\cdot}^*, L_{\cdot}^* - R_{\cdot}^*)$.
\end{remark}

\begin{corollary}
	Let $(A, \cdot)$ be a perm algebra and $r \in A \otimes A$ be a solution of the perm-YBE in $(A, \cdot)$ whose skew-symmetric part $\Theta$ is $(R, \ad)$-invariant.
	Define the following binary operations for all $x^*, y^* \in A^*$:
	\begin{align*}
		x^* \circ y^* =  - 2 L_\cdot^*(\Theta_{+}(x^*)) y^*, \;\;
		x^* \succ y^* = L_\cdot^*(r_{+}(x^*))y^*, \;\;
		x^* \prec y^* = R_\cdot^*(r_{+}(y^*))x^*.
	\end{align*}
	Then $(A^*, \circ, \succ, \prec)$ is a post-perm algebra.
	In particular, if $r_{+}$ is invertible, then there exists a post-perm algebra structure defined on $A$, whose associated perm algebra is precisely $(A, \cdot)$.
\end{corollary}
\begin{proof}
	By Proposition~\ref{prop:peqc}, $r_{+}$ is an $\mathcal{O}$-operator of weight $1$ on $(A, \cdot)$ associated to the $A$-bimodule perm algebra $(A^*, \circ_{+}, L_{\cdot}^*, L_{\cdot}^* - R_{\cdot}^*)$, where $\circ_{+}$ is defined by Eq.~\eqref{eq:dmd}.
	The conclusions now follow from Theorem~\ref{thm:o2pp} and Proposition~\ref{prop:ippeq}.
\end{proof}

%%%%%%%%%%%%%%%%%%%%%%%%%%%%%%%%%%%%%%%%%%%%%%%%%%%%%%%%%%%%%%%%%%%%%%%%%%%%%%%%
\subsection{Extended \texorpdfstring{$\mathcal{O}$}{O}-operators and extended perm-YBE on quadratic perm algebras}\label{sec:quadratic}

\begin{definition}
	A \textbf{quadratic perm algebra} is a triple $(A, \cdot, \mathcal{B})$, where $(A, \cdot)$ is a perm algebra and $\mathcal{B}$ is a nondegenerate skew-symmetric bilinear form on $A$ such that
	\begin{equation*}
		\mathcal{B}(a \cdot b, c) = \mathcal{B}(a, b \cdot c - c \cdot b), \;\;
		\forall a, b, c \in A.
	\end{equation*}
\end{definition}
For a quadratic perm algebra $(A, \cdot, \mathcal{B})$, the bilinear form satisfies $\mathcal{B}(a \cdot b, c ) = \mathcal{B}(b, a \cdot c)$ for all $a, b, c \in A$.

\begin{definition}
	Let $A$ be a vector space and $\mathcal{B}$ be a nondegenerate bilinear form on $A$.
	A linear map $T: A \rightarrow A$ is called {\bf self-adjoint} (resp. {\bf skew-adjoint}) with respect to $\mathcal{B}$ if for all $a, b \in A$
	\begin{equation*}
		\mathcal{B}(T(a), b) = \mathcal{B}(a, T(b)) \quad \text{(resp. $\mathcal{B}(T(a), b) = -\mathcal{B}(a, T(b))$)}.
	\end{equation*}
\end{definition}

Let $A$ be a vector space and $\mathcal{B}$ be a nondegenerate  bilinear form on $A$.
Denote by $I_\mathcal{B}: A^* \rightarrow A$ the induced linear isomorphism defined by
\begin{equation*}
	\langle I_\mathcal{B}^{-1}(a), b\rangle := \mathcal{B}(a, b), \quad \forall a,b \in A. 
\end{equation*}
Moreover, denote by $r_\mathcal{B}\in A \otimes A $ the 2-tensor form of $I_\mathcal{B}$,  that is,
\begin{equation*}
	\langle r_\mathcal{B}, x^* \otimes y^*\rangle := \langle   I_\mathcal{B}(x^*),y^*\rangle, \quad \forall  x^*,y^* \in A^*. 
\end{equation*}

\begin{lemma}\label{lem:sa2s}
	Let $(A, \cdot, \mathcal{B})$ be a quadratic perm algebra, and $T: A \to A$ be a linear map. 
	Then $T$ is self-adjoint (resp. skew-adjoint) with respect to $\mathcal{B}$ if and only if the 2-tensor form of $T I_\mathcal{B}: A^* \rightarrow A$ is skew-symmetric (resp. symmetric).
\end{lemma}
\begin{proof}
	It suffices to prove the self-adjoint case, as the skew-adjoint case follows by the same argument.
	Let $r$ be the tensor corresponding to $T I_{\mathcal{B}}$, i.e., $r_{+} = T I_{\mathcal{B}}$.
	Then, for all $x^*, y^* \in A^*$,
	\begin{align*}
		&\langle r + \tau(r), x^* \otimes y^* \rangle = \langle y^*, r_{+}(x^*) \rangle + \langle x^*, r_{+}(y^*) \rangle = \mathcal{B}(I_{\mathcal{B}}(y^*), r_{+}(x^*)) + \mathcal{B}(I_{\mathcal{B}}(x^*), r_{+}(y^*)) \\
		&=\mathcal{B}(I_{\mathcal{B}}(y^*), T I_{\mathcal{B}}(x^*)) + \mathcal{B}(I_{\mathcal{B}}(x^*), T I_{\mathcal{B}}(y^*)) = \mathcal{B}(I_{\mathcal{B}}(y^*), T I_{\mathcal{B}}(x^*)) - \mathcal{B}(T I_{\mathcal{B}}(y^*), I_{\mathcal{B}}(x^*)).
	\end{align*}
	Hence, $r$ is skew-symmetric if and only if $T$ is self-adjoint.
\end{proof}

\begin{lemma}\label{lem:eq4q}
	Let $(A, \cdot, \mathcal{B})$ be a quadratic perm algebra and $I_\mathcal{B}$ be the induced linear isomorphism by $\mathcal{B}$.
	Suppose that $S: A \rightarrow A$ is a linear map that is self-adjoint with respect to $\mathcal{B}$.
	Then the following conditions are equivalent.
	\begin{enumerate}
		\item\label{it:eq41}
		      $S$ is balanced associated to $(A, L_{\cdot}, R_{\cdot})$.

		\item\label{it:eq42}
		      $S$ is an $A$-bimodule homomorphism from $(A, L_{\cdot}, R_{\cdot})$ to $(A, L_{\cdot}, R_{\cdot})$.

		\item\label{it:eq43}
		      $P_S = S I_\mathcal{B}: A^* \rightarrow A$ is balanced associated to $(A^*, L_{\cdot}^*, L_{\cdot}^* - R_{\cdot}^*)$.

		\item\label{it:eq44}
		      $P_S = S I_\mathcal{B}: A^* \rightarrow A$ is an $A$-bimodule homomorphism from $(A^*, L_{\cdot}^*, L_{\cdot}^* - R_{\cdot}^*)$ to $(A, L_{\cdot}, R_{\cdot})$.
	\end{enumerate}
\end{lemma}
\begin{proof}
	Let $x^*, y^* \in A^*$ and $a, b, c \in A$.
	Then the following identities
	\begin{align*}
		 & \mathcal{B}(S(b \cdot c) - b \cdot S(c), a) = \mathcal{B}(b \cdot c, S(a)) - \mathcal{B}(b \cdot S(c), a) \\
		 &=  \mathcal{B}(b, c \cdot S(a) - S(a) \cdot c -  S(c) \cdot a + a \cdot S(c)), \\
		 & \mathcal{B}(S(b \cdot c) - S(b) \cdot c, a) = \mathcal{B}(b \cdot c, S(a)) - \mathcal{B}(S(b) \cdot c, a) =  \mathcal{B}(- b \cdot S(a) +  S(b) \cdot a, c),
	\end{align*}
	imply the equivalence between (\ref{it:eq41}) and (\ref{it:eq42}).
	Similarly, the following computation
	\begin{align*}
		& \langle L_{\cdot}^*(P_S(x^*))y^* - \ad_{\cdot}^*(P_S(y^*))x^*, a\rangle = \langle y^*, P_S(x^*) \cdot a  \rangle - \langle x^*, P_S(y^*) \cdot a - a \cdot P_S(y^*) \rangle                                                                      \\
		& = \mathcal{B}(I_\mathcal{B}(y^*), P_S(x^*) \cdot a) - \mathcal{B}(I_\mathcal{B}(x^*), P_S(y^*) \cdot a - a \cdot P_S(y^*)) \\
		&= \mathcal{B}(SI_\mathcal{B}(x^*) \cdot I_\mathcal{B}(y^*) - I_\mathcal{B}(x^*) \cdot SI_\mathcal{B}(y^*), a) ,
	\end{align*}
	establishes the equivalence between (\ref{it:eq41}) and (\ref{it:eq43}).
	Finally, identities
	\begin{align*}
		 & \langle y^*, P_S(L_{\cdot}^*(a)x^*) - a \cdot P_S(x^*)\rangle = \mathcal{B}( I_\mathcal{B}(y^*), P_S(L_{\cdot}^*(a)x^*)) - \mathcal{B}(I_\mathcal{B}(y^*), a \cdot P_S(x^*)) \\
		 & = \mathcal{B}(P_S(y^*), I_\mathcal{B}(L_{\cdot}^*(a)x^*)) + \mathcal{B}(a \cdot P_S(x^*), I_\mathcal{B}(y^*) ) = -\langle x^*, a \cdot P_S(y^*) \rangle + \mathcal{B}(a \cdot P_S(x^*), I_\mathcal{B}(y^*) ) \\
		 &=-\mathcal{B}(I_\mathfrak{B}(x^*), a \cdot P_S(y^*)) +  \mathcal{B}(a \cdot P_S(x^*), I_\mathcal{B}(y^*) ) \\
		 &= \mathfrak{B}(a, SI_\mathcal{B}(y^*) \cdot I_\mathcal{B}(x^*) - I_\mathcal{B}(x^*) \cdot SI_\mathcal{B}(y^*) - I_\mathcal{B}(y^*) \cdot P_S(x^*) +  P_S(x^*) \cdot I_\mathcal{B}(y^*)), \\
		 & \langle y^*, P_S(\ad_{\cdot}^*(a)x^*) - P_S(x^*) \cdot a \rangle = \mathcal{B}( I_\mathcal{B}(y^*), P_S(\ad_{\cdot}^*(a)x^*)) - \mathcal{B}(I_\mathcal{B}(y^*), P_S(x^*) \cdot a) \\
		 & = \mathcal{B}(P_S(y^*), I_\mathcal{B}(\ad_{\cdot}^*(a)x^*)) - \mathcal{B}(P_S(x^*) \cdot I_\mathcal{B}(y^*), a) = -\langle \ad_{\cdot}^*(a)x^*, P_S(y^*) \rangle - \mathcal{B}(P_S(x^*) \cdot I_\mathcal{B}(y^*), a) \\
		 &= -\mathcal{B}(I_\mathcal{B}(x^*), \ad_{\cdot}(a)(P_S(y^*)) ) - \mathcal{B}(P_S(x^*) \cdot I_\mathcal{B}(y^*), a) = \mathcal{B}(I_\mathcal{B}(x^*) \cdot SI_\mathcal{B}(y^*) - SI_\mathcal{B}(x^*) \cdot I_\mathcal{B}(y^*), a),
	\end{align*}
	show the equivalence between (\ref{it:eq41}) and (\ref{it:eq44}).
	The proof is complete.
\end{proof}

\begin{proposition}\label{prop:ajo2s}
	Let $(A, \cdot, \mathcal{B})$ be a quadratic perm algebra, $I_\mathcal{B}$ be the induced linear isomorphism by $\mathcal{B}$, and $T, S: A \rightarrow A$ be linear maps.
	Then $T$ is an extended $\mathcal{O}$-operator of weight $0$ with extension $S$ of mass $(\kappa, 0)$ on $(A, \cdot)$ associated to $(A, L_{\cdot}, R_{\cdot})$ if and only if $P_T=T I_\mathcal{B}: A^* \rightarrow A$ is an extended $\mathcal{O}$-operator of weight $0$ with extension $P_S = S I_\mathcal{B}: A^* \rightarrow A$ of mass $(\kappa, 0)$ on $(A, \cdot)$ associated to $(A^*, L_{\cdot}^*, L_{\cdot}^* - R_{\cdot}^*)$.	
\end{proposition}
\begin{proof}
	For all $a, b \in A$ and $x^* \in A^*$, we have
	\begin{align*}
		&\mathcal{B}(I_\mathcal{B}(L_{\cdot}^*(a)x^*), b) = \langle L_{\cdot}^*(a)x^*, b\rangle = \langle x^*, a \cdot b\rangle = \mathcal{B}(I_\mathcal{B}(x^*), a \cdot b) = \mathcal{B}(a \cdot I_\mathcal{B}(x^*), b), \\
		&\mathcal{B}(I_\mathcal{B}(\ad_{\cdot}^*(a)x^*), b) = \langle \ad_{\cdot}^*(a)x^*, b\rangle = \langle x^*, a \cdot b - b \cdot a\rangle = \mathcal{B}(I_\mathcal{B}(x^*), a \cdot b - b \cdot a ) = \mathcal{B}(I_\mathcal{B}(x^*) \cdot a, b).
	\end{align*}
	That is, 
	\begin{equation*}
		I_\mathcal{B}(L_{\cdot}^*(a)x^*) = a \cdot I_\mathcal{B}(x^*), \;\;
		I_\mathcal{B}(\ad_{\cdot}^*(a)x^*) = I_\mathcal{B}(x^*) \cdot a, \;\; \forall a \in A, x^* \in A^*.
	\end{equation*}
	Therefore,
	\begin{align*}
		 & P_T(x^*) \cdot P_T(y^*) - P_T\big(L_{\cdot}^*(P_T(x^*))y^* + \ad_{\cdot}^*(P_T(y^*))x^* \big) - \kappa P_S(x^*) \cdot P_S(y^*) \\
		 & =TI_\mathcal{B}(x^*) \cdot TI_\mathcal{B}(y^*) - TI_\mathcal{B}\big(L_{\cdot}^*(TI_\mathcal{B}(x^*))y^* + \ad_{\cdot}^*(TI_\mathcal{B}(y^*))x^* \big) - \kappa SI_\mathcal{B}(x^*) \cdot SI_\mathcal{B}(y^*)                          \\
		 & = TI_\mathcal{B}(x^*) \cdot TI_\mathcal{B}(y^*) - T\big(L_{\cdot}(TI_\mathcal{B}(x^*))I_\mathcal{B}(y^*) + R_{\cdot}(TI_\mathcal{B}(y^*))I_\mathcal{B}(x^*) \big) - \kappa SI_\mathcal{B}(x^*) \cdot SI_\mathcal{B}(y^*).
	\end{align*}
	Hence, $T$ is an extended $\mathcal{O}$-operator of weight $0$ with extension $S$ of mass $(\kappa, 0)$ on $(A, \cdot)$ associated to $(A, L_{\cdot}, R_{\cdot})$ if and only if $P_T=T I_\mathcal{B}: A^* \rightarrow A$ is an extended $\mathcal{O}$-operator of weight $0$ with extension $P_S = S I_\mathcal{B}: A^* \rightarrow A$ of mass $(\kappa, 0)$ on $(A, \cdot)$ associated to $(A^*, L_{\cdot}^*, L_{\cdot}^* - R_{\cdot}^*)$.
\end{proof}

\begin{corollary}\label{coro:exoqs}
	Let $(A, \cdot, \mathcal{B})$ be a quadratic perm algebra, $I_\mathcal{B}$ be the induced linear isomorphism by $\mathcal{B}$, and $T, S: A \rightarrow A$ be linear maps.
	Suppose that $S$ is balanced associated to $(A, L_{\cdot}, R_{\cdot})$ and self-adjoint with respect to $\mathcal{B}$, and $T$ is skew-adjoint with respect to $\mathcal{B}$. 
	Then the 2-tensor form of $T I_\mathcal{B} + S I_\mathcal{B}$ (resp $T I_\mathcal{B} - S I_\mathcal{B}$) is a solution of the extended perm-YBE of mass $\frac{\kappa+1}{4}$ in $(A, \cdot)$ if and only if $T$ is an extended $\mathcal{O}$-operator of weight $0$ with extension $S$ of mass $(\kappa, 0)$ on $(A, \cdot)$ associated to $(A, L_{\cdot}, R_{\cdot})$.
	In particular, the following statements hold.
	\begin{enumerate}
		\item
		The 2-tensor form of $T I_\mathcal{B} + S I_\mathcal{B}$ (resp. $T I_\mathcal{B} - S I_\mathcal{B}$) is a solution of the perm-YBE in $(A, \cdot)$ if and only if $T$ is an extended $\mathcal{O}$-operator of weight $0$ with extension $S$ of mass $(-1, 0)$ on $(A, \cdot)$ associated to $(A, L_{\cdot}, R_{\cdot})$.
		
		\item
		The 2-tensor form of $T I_\mathcal{B}$ is a solution of the perm-YBE in $(A, \cdot)$ if and only if $T$ is a Rota-Baxter operator of weight $0$ on $(A, \cdot)$.
	\end{enumerate}
\end{corollary}
\begin{proof}
	By Lemma~\ref{lem:sa2s}, the 2-tensor form of $T I_\mathcal{B}$ is symmetric, while that of $S I_\mathcal{B}$ is skew-symmetric.
	Since $S$ is balanced associated to $(A, L_{\cdot}, R_{\cdot})$, it follows from Lemma~\ref{lem:eq4q} and \ref{lem:sinvb} that the 2-tensor form of $S I_\mathcal{B}$ is $(R, \ad)$-invariant.
	Then, by Theorem~\ref{thm:extended-perm-YBE-o} and Corollary~\ref{coro:trr}, the 2-tensor form of $T I_\mathcal{B} + S I_\mathcal{B}$ (resp. $T I_\mathcal{B} - S I_\mathcal{B}$) is a solution of the extended perm-YBE of mass $\frac{\kappa+1}{4}$ in $(A, \cdot)$ if and only if $T I_\mathcal{B}: A^* \rightarrow A$ is an extended $\mathcal{O}$-operator of weight $0$ with extension $S I_\mathcal{B}$ of mass $(\kappa, 0)$ on $(A, \cdot)$ associated to $(A^*, L_{\cdot}^*, L_{\cdot}^* - R_{\cdot}^*)$.
	Therefore, Proposition~\ref{prop:ajo2s} implies that the 2-tensor form of $T I_\mathcal{B} + S I_\mathcal{B}$ (resp. $T I_\mathcal{B} - S I_\mathcal{B}$) is a solution of the extended perm-YBE of mass $\frac{\kappa+1}{4}$ in $(A, \cdot)$ if and only if $T$ is an extended $\mathcal{O}$-operator of weight $0$ with extension $S$ of mass $(\kappa, 0)$ on $(A, \cdot)$ associated to $(A, L_{\cdot}, R_{\cdot})$.
	Finally, the particular cases are obtained by taking $\kappa = -1$ or $S = 0$ respectively.
\end{proof}

\begin{corollary}
	Let $(A, \cdot, \mathcal{B})$ be a quadratic perm algebra, $I_\mathcal{B}: A^* \rightarrow A$ be the induced linear isomorphism by $\mathcal{B}$ and $r \in A \otimes A$.
	Writing $r$ as $r = \Lambda + \Theta$ with $\Lambda \in \mathrm{Sym}^2(A)$ and $\Theta \in \mathrm{Alt}^2(A)$.
	Suppose that $\Theta$ is $(R, \ad)$-invariant.
	Then $r$ is a solution of the extended perm-YBE of mass $\frac{\kappa+1}{4}$ if and only if $\Lambda_{+}I_\mathcal{B}^{-1}: A \rightarrow A$ is an extended $\mathcal{O}$-operator of weight $0$ with extension $\Theta_{+}I_\mathcal{B}^{-1}: A \rightarrow A$ of mass $(\kappa, 0)$ on $(A, \cdot)$ associated to $(A, L_{\cdot}, R_{\cdot})$.
	In particular, the following statements hold.
	\begin{enumerate}
		\item
		      $r$ is a solution of the perm-YBE in $(A, \cdot)$ if and only if $\Lambda_{+}I_\mathcal{B}^{-1}: A \rightarrow A$ is an extended $\mathcal{O}$-operator of weight $0$ with extension $\Theta_{+}I_\mathcal{B}^{-1}: A \rightarrow A$ of mass $(-1, 0)$ on $(A, \cdot)$ associated to $(A, L_{\cdot}, R_{\cdot})$.

		\item
		      $\Lambda$ is a solution of the perm-YBE in $(A, \cdot)$ if and only if $\Lambda_{+}I_\mathcal{B}^{-1}: A \rightarrow A$ is a Rota-Baxter operator of weight 0 on $(A, \cdot)$.
	\end{enumerate}
\end{corollary}
\begin{proof}
	By Lemma~\ref{lem:sa2s}, $\Lambda_{+}I_\mathcal{B}^{-1}$ is skew-adjoint with respect to $\mathcal{B}$, and $\Theta_{+}I_\mathcal{B}^{-1}$ is self-adjoint with respect to $\mathcal{B}$.
	Moreover, since $\Theta$ is $(R, \ad)$-invariant, it follows from Lemma~\ref{lem:sinvb} and \ref{lem:eq4q} that $\Theta_{+}I_\mathcal{B}^{-1}:A \rightarrow A$ is balanced associated to $(A, L_{\cdot}, R_{\cdot})$.
	The desired conclusions now follow from Corollary~\ref{coro:exoqs}.
\end{proof}

\begin{example}
	Let $(A, \cdot)$ be the 2-dimensional perm algebra with basis $\{e_1, e_2\}$ whose non-zero multiplication is defined by
	\begin{align*}
		e _1\cdot e_1 & = e_1,\quad e_1\cdot e_2 = e_2.
	\end{align*}
	Define a linear map $T: A \to A$ by
	\begin{equation*}
		T(e_1) = \gamma e_1, \;\; T(e_2) = - \gamma e_2.
	\end{equation*}
	Then $T$ is an extended $\mathcal{O}$-operator of weight $0$ with extension $S=\id$ of mass $(-\gamma^2, 0)$ on $(A, \cdot)$ associated to $(A, L_{\cdot}, R_{\cdot})$.
	Let $\mathcal{B}$ be the bilinear form defined by
	\begin{equation*}
		\mathcal{B}(e_1, e_1) = \mathcal{B}(e_2, e_2) = 0, \;\;
		\mathcal{B}(e_1, e_2) = -\mathcal{B}(e_2, e_1) = 1.
	\end{equation*}
	Then $(A, \cdot, \mathcal{B})$ is a quadratic perm algebra.
	Moreover $T$ is skew-adjoint with respect to $\mathcal{B}$.
	Therefore, the 2-tensor form of $T I_\mathcal{B} \pm S I_\mathcal{B}$
	\begin{equation*}
		(\gamma - 1) e_1 \otimes e_2 + (\gamma  + 1)e_2 \otimes e_1,  (\gamma + 1) e_1 \otimes e_2 + (\gamma  - 1)e_2 \otimes e_1
	\end{equation*}
	are solutions of the extended perm-YBE of mass $\frac{1-\gamma^2}{4}$ in $(A, \cdot)$.
	In particular, 
	\begin{equation*}
		\pm 2 e_1 \otimes e_2, \pm 2 e_2 \otimes e_1
	\end{equation*}
	are solutions of the perm-YBE in $(A, \cdot)$.
\end{example}

\subsection{Extended \texorpdfstring{$\mathcal{O}$}{O}-operators and extended perm-YBE on semi-direct product}
\label{sec:extension}

\begin{lemma}\label{lem:b2bsm}
	Let $(A, \cdot)$ be a perm algebra, $(V, \frakl, \frakr)$ a bimodule of $(A, \cdot)$, and $(\mathfrak{A}=A \ltimes_{\frakl^*, \frakl^* - \frakr^*} V^*, \bullet)$ be the semi-direct product of $(A, \cdot)$ and $(V^*, \frakl^*, \frakl^* - \frakr^*)$.
	Let $S: V \rightarrow A$ be a linear map.
	Then $\mathcal{S} := r^{S}_{+} - r^{S}_{-}: \mathfrak{A}^* \rightarrow \mathfrak{A}$ is a balanced $\mathfrak{A}$-bimodule homomorphism from $(\mathfrak{A}^*, L_{\bullet}^*, L_{\bullet}^* - R_{\bullet}^*)$ to $(\mathfrak{A}, L_{\bullet}, R_{\bullet})$ if and only if $S$ is a balanced $A$-bimodule homomorphism from $(V, \frakl, \frakr)$ to $(A, L_{\cdot}, R_{\cdot})$.
\end{lemma}
\begin{proof}
	Note that $r^S - \tau(r^S)$ is skew-symmetric, by Lemma~\ref{lem:sinvb}, $r^{S}_{+} - r^{S}_{-}: \mathfrak{A}^* \rightarrow \mathfrak{A}$ is a balanced $\mathfrak{A}$-bimodule homomorphism from $(\mathfrak{A}^*, L_{\bullet}^*, L_{\bullet}^* - R_{\bullet}^*)$ to $(\mathfrak{A}, L_{\bullet}, R_{\bullet})$ if and only if it is balanced associated to $(\mathfrak{A}^*, L_{\bullet}^*, L_{\bullet}^* - R_{\bullet}^*)$.
	It then suffices to prove that $\mathcal{S}$ is balanced associated to $(\mathfrak{A}^*, L_{\bullet}^*, L_{\bullet}^* - R_{\bullet}^*)$ if and only if $S$ is a balanced $A$-bimodule homomorphism from $(V, \frakl, \frakr)$ to $(A, L_{\cdot}, R_{\cdot})$.
	For all $x^* \in A^*$ and $u \in V$, we have
	\begin{equation*}
		(r^{S}_{+} - r^{S}_{-})(x^* + u) = -S(u) + S^*(x^*).
	\end{equation*}
	Let $x^*, y^* \in A^*$, $a \in A$, $\xi^* \in V^*$ and $u, v \in V$, we have
	\begin{eqnarray*}
		 && \langle -L_{\bullet}^*((r^{S}_{+} - r^{S}_{-})(x^* + u)) (y^* + v) + \ad_{\bullet}^*((r^{S}_{+} - r^{S}_{-})(y^* + v)) (x^* + u), a + \xi^*\rangle \\
		 && = \langle -L_{\bullet}^*(-S(u) + S^*(x^*)) (y^* + v) + \ad_{\bullet}^*(-S(v) + S^*(y^*)) (x^* + u), a + \xi^* \rangle                                \\
		 && = \langle y^* + v, (S(u) - S^*(x^*)) \bullet (a + \xi^*)\rangle \\
		 &&\quad + \langle x^* + u, (-S(v) + S^*(y^*)) \bullet (a + \xi^*) - (a + \xi^*) \bullet (-S(v) + S^*(y^*)) \rangle \\
		 &&= \langle y^*, S(u) \cdot a \rangle + \langle \frakl^*(S(u))\xi^* - (\frakl^* - \frakr^*)(a)S^*(x^*), v\rangle \\
		 &&\quad + \langle x^* + u, - S(v) \cdot a + a \cdot S(v) - \frakr^*(a)S^*(y^*) - \frakr^*(S(v))\xi^* \rangle \\
		 && = \langle y^*, S(u) \cdot a - S(
		\frakr(a)u) \rangle  + \langle \xi^*, \frakl(S(u))v - \frakr(S(v))u \rangle \\
		&& + \langle x^*, - S(v) \cdot a + a \cdot S(v) - S((\frakl - \frakr)(a)v)\rangle.
	\end{eqnarray*}
	Therefore, $r^{S}_{+} - r^{S}_{-}: \mathfrak{A}^* \rightarrow \mathfrak{A}$ is balanced associated to $(\mathfrak{A}^*, L_{\bullet}^*, L_{\bullet}^* - R_{\bullet}^*)$ if and only if $S$ is a balanced $A$-bimodule homomorphism from $(V, \frakl, \frakr)$ to $(A, L_{\cdot}, R_{\cdot})$.
	The proof is complete.
\end{proof}

\begin{theorem}\label{thm:eo2smp}
	Let $(A, \cdot)$ be a perm algebra, $(V, \frakl, \frakr)$ be a bimodule of $(A, \cdot)$, and $(\mathfrak{A}=A \ltimes_{\frakl^*, \frakl^* - \frakr^*} V^*, \bullet)$ be the semi-direct product of $(A, \cdot)$ and $(V^*, \frakl^*, \frakl^* - \frakr^*)$.
	Let $T, S: V \rightarrow A$ be linear maps.
	Suppose that $S$ is a balanced $A$-bimodule homomorphism from $(V, \frakl, \frakr)$ to $(A, L_{\cdot}, R_{\cdot})$.
	Then $\mathcal{T} := r_{+}^{T} + r_{-}^{T}: \mathfrak{A}^* \rightarrow \mathfrak{A}$ is an extended $\mathcal{O}$-operator of weight 0 with extension $\mathcal{S} := r_{+}^{S} - r_{-}^{S}$ of mass $(\kappa, 0)$ on $(\mathfrak{A}, \bullet)$ associated $(\mathfrak{A}^*, L_{\bullet}^*, L_{\bullet}^* - R_{\bullet}^*)$ if and only if $T: V \rightarrow A$ is an extended $\mathcal{O}$-operator of weight 0 with extension $S$ of mass $(\kappa, 0)$ on $(A, \cdot)$ associated to $(V, \frakl, \frakr)$.
\end{theorem}
\begin{proof}
	For all $x^* \in A^*$ and $v \in V$, we have
	\begin{equation*}
		\mathcal{T}(x^* + u) =  T(u) + T^*(x^*).
	\end{equation*}
	For all $x^*, y^*, z^* \in A^*$ and $u, v, w \in V$, we have
	\begin{eqnarray*}
		 && \langle \mathcal{T}(x^* + u) \bullet \mathcal{T}(y^* + v) - \kappa \mathcal{S}(x^*+u) \bullet \mathcal{S}(y^*+v), z^*+ w\rangle                             \\
		 && \quad - \langle \mathcal{T}\big( L_{\bullet}^*(\mathcal{T}(x^* + u))(y^*+v) + \ad_{\bullet}^*(\mathcal{T}(y^* + v))(x^*+u)\big), z^*+ w\rangle \\
		 && =\langle (T(u) + T^*(x^*)) \bullet ( T(v) + T^*(y^*)) - \kappa (-S(u) + S^*(x^*)) \bullet (-S(v) + S^*(y^*)), z^*+ w\rangle \\
		 && \quad - \langle \mathcal{T}\big( L_{\bullet}^*(T(u) + T^*(x^*))(y^*+v) + \ad_{\bullet}^*(T(v) + T^*(y^*))(x^*+u)\big), z^*+ w \rangle \\
		 && =\langle z^*, T(u) \cdot T(v) - \kappa S(u) \cdot S(v) \rangle \\
		 && \quad  + \langle \frakl^*(T(u))T^*(y^*) + (\frakl^* - \frakr^*)(T(v))T^*(x^*) + \kappa \frakl^*(S(u))S^*(y^*) + \kappa (\frakl^* - \frakr^*)(S(v))S^*(x^*), w\rangle                                 \\
		 && \quad - \langle L_{\bullet}^*(T(u) + T^*(x^*) )(y^*+v) + \ad_{\bullet}^*(T(v) + T^*(y^*) )(x^*+u), \mathcal{T}(z^* + w)\rangle                    \\
		 && =\langle z^*, T(u) \cdot T(v) - \kappa S(u) \cdot S(v) \rangle \\
		 && \quad + \langle y^*, T(\frakl(T(u))w) + \kappa S(\frakl(S(u))w)\rangle + \langle x^*, T((\frakl - \frakr)(T(v))w) + \kappa S((\frakl - \frakr)(S(v))w) \rangle                                  \\
		 && \quad - \langle y^* + v, (T(u) + T^*(x^*)) \bullet (T(w) + T^*(z^*)) \rangle \\
		 && \quad - \langle x^* + u, (T(v) + T^*(y^*)) \bullet (T(w) + T^*(z^*)) - ( T(w) + T^*(z^*)) \bullet (T(v) + T^*(y^*)) \rangle \\
		 && =\langle z^*, T(u) \cdot T(v) - \kappa S(u) \cdot S(v) \rangle \\
		 && \quad + \langle y^*, T(\frakl(T(u))w) + \kappa S(\frakl(S(u))w)\rangle + \langle x^*, T((\frakl - \frakr)(T(v))w) + \kappa S((\frakl - \frakr)(S(v))w) \rangle \\
		 && \quad - \langle y^* + v, T(u) \cdot T(w) + \frakl^*(T(u))(T^*(z^*)) + (\frakl^* - \frakr^*)(T(w))(T^*(x^*)) \rangle                                                                   \\
		 && \quad - \langle x^* + u, T(v) \cdot T(w)  - \frakr^*(T(w))(T^*(y^*)) - T(w) \cdot T(v)  + \frakr^*(T(v))(T^*(z^*)) \rangle                                                                   \\
		 && =\langle z^*, T(u) \cdot T(v) - T(\frakl(T(u))v) - T(\frakr(T(v))u) - \kappa S(u) \cdot S(v) \rangle                                                                         \\
		 && \quad + \langle y^*, T(\frakl(T(u))w) + T(\frakr(T(w))u) - T(u) \cdot T(w) + \kappa S(\frakl(S(u))w) \rangle                                                               \\
		 && \quad + \langle x^*, - T((\frakl - \frakr)(T(w))v) + T(w) \cdot T(v) - \kappa S(\frakr(S(v))w) \rangle                                                                \\
		 && \quad + \langle x^*, T((\frakl - \frakr)(T(v))w) -T(v) \cdot T(w) + \kappa S(\frakl(S(v))w) \rangle                                                                \\
		 && =\langle z^*, T(u) \cdot T(v) - T(\frakl(T(u))v) - T(\frakr(T(v))u) - \kappa S(u) \cdot S(v) \rangle                                                                         \\
		 && \quad - \langle y^*, - T(\frakl(T(u))w) - T(\frakr(T(w))u) + T(u) \cdot T(w) - \kappa S(u) \cdot S(w) \rangle                                                               \\
		 && \quad + \langle x^*,  - T((\frakl - \frakr)(T(w))v) + T(w) \cdot T(v)  - \kappa S(w) \cdot S(v) \rangle \\
		 && \quad + \langle x^*, T((\frakl - \frakr)(T(v))w)  -T(v) \cdot T(w)  + \kappa S(v) \cdot S(w)  \rangle,
	\end{eqnarray*}
	Note that
	\begin{align*}
		&T((\frakl - \frakr)(T(v))w)  -T(v) \cdot T(w)  + \kappa S(v) \cdot S(w) - T((\frakl - \frakr)(T(w))v) + T(w) \cdot T(v)  - \kappa S(w) \cdot S(v) \\
		&=\big( T(w) \cdot T(v) - T(\frakl(T(w))v) - T(\frakr(T(v))w) - \kappa S(w) \cdot S(v) \big) \\
		&\quad - \big( T(v) \cdot T(w) - T(\frakl(T(v))w) - T(\frakr(T(w))v) - \kappa S(v) \cdot S(w) \big).
	\end{align*}
	Therefore, $\mathcal{T}$ is an extended $\mathcal{O}$-operator of weight 0 with extension $\mathcal{S}$ of mass $(\kappa, 0)$ on $(\mathfrak{A}, \bullet)$ associated $(\mathfrak{A}^*, L_{\bullet}^*, L_{\bullet}^* - R_{\bullet}^*)$ if and only if $T: V \rightarrow A$ is an extended $\mathcal{O}$-operator of weight 0 with extension $S$ of mass $(\kappa, 0)$ on $(A, \cdot)$ associated to $(V, \frakl, \frakr)$.
\end{proof}

\begin{corollary}\label{cor:tsc}
	Let $(A, \cdot)$ be a perm algebra, $(V, \frakl, \frakr)$ be a bimodule of $(A, \cdot)$, and $(\mathfrak{A} = A \ltimes_{\frakl^*, \frakl^* - \frakr^*} V^*, \bullet)$ be the semi-direct product of $(A, \cdot)$ and $(V^*, \frakl^*, \frakl^* - \frakr^*)$.
	Suppose that $S: V \rightarrow A$ is a balanced $A$-bimodule homomorphism from $(V, \frakl, \frakr)$ to $(A, L_{\cdot}, R_{\cdot})$.
	Then $T: V \rightarrow A$ is an extended $\mathcal{O}$-operator of weight 0 with extension $S$ of mass $(\kappa, 0)$ on $(A, \cdot)$ associated to $(V, \frakl, \frakr)$ if and only if $(r^T + \tau(r^T)) + (r^S - \tau(r^S))$ (resp. $(r^T + \tau(r^T)) - (r^S - \tau(r^S))$) is a solution of the extended perm-YBE of mass $\frac{\kappa+1}{4}$ in $(A \ltimes_{\frakl^*, \frakl^* - \frakr^*} V^*, \bullet)$.
	In particular, the following statements hold.
	\begin{enumerate}
		\item
		      $T$ is an extended $\mathcal{O}$-operator of weight 0 with extension $S$ of mass $(-1, 0)$ on $(A, \cdot)$ associated to $(V, \frakl, \frakr)$ if and only if $(r^T + \tau(r^T)) + (r^S - \tau(r^S))$ (resp. $(r^T + \tau(r^T)) - (r^S - \tau(r^S))$) is a solution of the perm-YBE in $(A \ltimes_{\frakl^*, \frakl^* - \frakr^*} V^*, \bullet)$.

		\item
		      $T$ is an $\mathcal{O}$-operator of weight 0 associated to $(V, \frakl, \frakr)$ if and only if $r^T + \tau(r^T)$ is a symmetric solution of the perm-YBE in $(A \ltimes_{\frakl^*, \frakl^* - \frakr^*} V^*, \bullet)$.
	\end{enumerate}
\end{corollary}
\begin{proof}
	By Lemma~\ref{lem:b2bsm} and \ref{lem:sinvb}, $r^S - \tau(r^S)$ is $(R, \ad)$-invariant.
	Applying Theorem~\ref{thm:extended-perm-YBE-o} and Corollary~\ref{coro:trr}, we get that $(r^T + \tau(r^T)) + (r^S - \tau(r^S))$ (resp. $(r^T + \tau(r^T)) - (r^S - \tau(r^S))$) is a solution of the extended perm-YBE of mass $\frac{\kappa+1}{4}$ in $(A \ltimes_{\frakl^*, \frakl^* - \frakr^*} V^*, \bullet)$ if and only if $r_{+}^{T} + r_{-}^{T}: \mathfrak{A}^* \rightarrow \mathfrak{A}$ is an extended $\mathcal{O}$-operator of weight 0 with extension $r_{+}^{S} - r_{-}^{S}$ of mass $(\kappa, 0)$ on $(\mathfrak{A}, \bullet)$ associated $(\mathfrak{A}^*, L_{\bullet}^*, L_{\bullet}^* - R_{\bullet}^*)$.
	By Theorem~\ref{thm:eo2smp}, the latter condition is equivalent to $T: V \rightarrow A$ being an extended $\mathcal{O}$-operator of weight 0 with extension $S$ of mass $(\kappa, 0)$ on $(A, \cdot)$ associated to $(V, \frakl, \frakr)$.
	Thus, we obtain the desired equivalence.
	Particular cases follow by taking $\kappa = -1$ or $S = 0$.
\end{proof}

\begin{corollary}
	Let $(A, \cdot)$ be a perm algebra and  $(\mathfrak{A} = A \ltimes_{L_{\cdot}^*, L_{\cdot}^* - R_{\cdot}^*} A^*, \bullet)$ be the semi-direct product of $(A, \cdot)$ and $(A^*, L_{\cdot}^*, L_{\cdot}^* - R_{\cdot}^*)$.
	Let $T: A \rightarrow A$ be a linear map.
	Then
	\begin{enumerate}
		\item\label{it:c1}
		      $T$ is an extended $\mathcal{O}$-operator of weight 0 with extension $\id$ of mass $(-1,0)$ on $(A, \cdot)$ associated to $(V, \frakl, \frakr)$ $(A, L_{\cdot}, R_{\cdot})$ if and only if $r^T + \tau(r^T) + (r^\id - \tau(r^\id))$ (resp. $r^T + \tau(r^T) - (r^\id - \tau(r^\id))$) is a solution of the perm-YBE in $(A \ltimes_{L_{\cdot}^*, L_{\cdot}^* - R_{\cdot}^*} A^*, \bullet)$.

		\item\label{it:c2}
		      $T$ is a Rota-Baxter operator of nonzero weight $\lambda$ if and only if $\frac{2}{\lambda}(r^T + \tau(r^T)) + 2r^\id$ (resp. $\frac{2}{\lambda}(r^T + \tau(r^T)) + 2\tau(r^\id)$)  is a solution of the perm-YBE in $(A \ltimes_{L_{\cdot}^*, L_{\cdot}^* - R_{\cdot}^*} A^*, \bullet)$.
	\end{enumerate}
\end{corollary}
\begin{proof}
	(\ref{it:c1}). 
	The conclusion follows from Corollary~\ref{cor:tsc} by taking $S = \id$.
	
	(\ref{it:c2}).
	Note that $T: A \to A$ is a Rota-Baxter operator of weight $\lambda \neq 0$ if and only if $\frac{2}{\lambda}T + \id$ is an extended $\mathcal{O}$-operator of weight $0$ with extension $\id$ of mass $(-1, 0)$ on $(A, \cdot)$ associated to $(A, L_{\cdot}, R_{\cdot})$.
	Therefore, (\ref{it:c2}) follows from (\ref{it:c1}).
\end{proof}

\begin{example}
	Let $(A = \mathbf{k}, \cdot)$ be the 1-dimensional perm algebra with the usual scalar multiplication.
	Let $(\mathfrak{A} = A \ltimes_{L_{\cdot}^*, L_{\cdot}^* - R_{\cdot}^*} A^*, \bullet)$ be the semi-direct product of $(A, \cdot)$ and $(A^*, L_{\cdot}^*, L_{\cdot}^* - R_{\cdot}^*)$, that is,
	\begin{equation*}
		(a+\xi^*) \bullet (b + \eta^*) = ab + a\eta^*, \;\;
		\forall a, b \in A, \; \xi^*, \eta^* \in V^*. 
	\end{equation*}
	Take $T: A \rightarrow A$ defined by $T(a)= \gamma a$ for all $a \in A$, where $\gamma \in \mathbf{k}$ is fixed.
	Then $T$ is an extended $\mathcal{O}$-operator of weight $0$ with $S=\id$ of mass $(-\gamma^2, 0)$ associated to $(A, L_\cdot, R_\cdot)$.
	Let $e$ be the basis of $A$ and $e^*$ be the dual basis of $A^*$.
	Then Corollary~\ref{cor:tsc} shows that both
	\begin{equation*}
		(\gamma + 1)e\otimes e^* + (\gamma - 1)e^* \otimes e  \quad\text{and}\quad (\gamma - 1)e\otimes e^* + (\gamma + 1)e^* \otimes e
	\end{equation*}
	are solutions of the extended perm-YBE of mass $\frac{1-\gamma^2}{4}$ in $(\mathfrak{A} = A \ltimes_{L_{\cdot}^*, L_{\cdot}^* - R_{\cdot}^*} A^*, \bullet)$.
	In particular, taking $\gamma = \pm 1$, we have that
	\begin{equation*}
		\pm 2 e \otimes e^*, \pm 2 e^* \otimes e 
	\end{equation*}
	are solutions of the perm-YBE in $(\mathfrak{A} = A \ltimes_{L_{\cdot}^*, L_{\cdot}^* - R_{\cdot}^*} A^*, \bullet)$. 
\end{example}

%%%%%%%%%%%%%%%%%%%%%%%%%%%%%%%%%%%%%%%%%%%%%%%%%%%%%%%%%%%%%%%%%%%%%%%%%%%%%%%%
%%%%%%%%%%%%%%%%%%%%%%%%%%%%%%%%%%%%%%%%%%%%%%%%%%%%%%%%%%%%%%%%%%%%%%%%%%%%%%%%
%%%%%%%%%%%%%%%%%%%%%%%%%%%%%%%%%%%%%%%%%%%%%%%%%%%%%%%%%%%%%%%%%%%%%%%%%%%%%%%%

\bigskip

\noindent {\bf Data Availability.}
Data sharing is not applicable to this article as no new data were created or analyzed during this study.

\noindent {\bf Declarations.}
The authors have no competing interests to declare.

%%%%%%%%%%%%%%%%%%%%%%%%%%%%%%%%%%%%%%%%%%%%%%%%%%%%%%%%%%%%%%%%%%%%%%%%%%%%%%%%

%\bibliographystyle{camsplain} % latex makebst
%\bibliography{ref.bib}

\end{document}